\documentclass[12pt,reqno]{amsart}

\usepackage{amssymb, amsmath, amsfonts, latexsym}
\usepackage{enumerate}
\usepackage{color}
\newtheorem{thm}{Theorem}[]
\newtheorem{lem}{Lemma}[section]

\newtheorem{rmk}{Remark}[section]
\theoremstyle{definition}

\numberwithin{equation}{section} \theoremstyle{remark}

\title[Deviation for extremals of large chiral non-Hermitian ]{\bf  Deviation and moderate deviation for extremal eigenvalues of large Chiral non-Hermitian random matrices}

\author{Yutao Ma}
\address{Yutao MA\\ School of Mathematical Sciences $\&$ Laboratory  of Mathematics and Complex Systems of Ministry of Education, Beijing Normal University, 100875 Beijing, China.} 
\thanks{The research of Yutao Ma was supported in part by NSFC 12171038, 11871008 and 985 Projects.}
\email{mayt@bnu.edu.cn}

\author{S\MakeLowercase{iyu} W\MakeLowercase{ang}}
\address{Siyu Wang\\ School of Mathematical Sciences $\&$ Laboratory of Mathematics and Complex Systems of Ministry of Education, Beijing Normal University, 100875 Beijing, China.}
\email{wang\_siyu@mail.bnu.edu.cn}

\begin{document}
\maketitle

\begin{abstract}

Consider the chiral non-Hermitian random matrix ensemble with parameters $n$ and $v,$ and let $(\zeta_i)_{1\le i\le n}$ be its $n$ eigenvalues with positive $x$-coordinate. In this paper, we establish deviation probabilities and moderate deviation probabilities for the spectral radius $(n/(n+v))^{1/2}\max_{1\le i\le n}|\zeta_i|^2,$ as well as $(n/(n+v))^{1/2}\min_{1\le i\le n}|\zeta_i|^2.$    
\end{abstract} 


{\bf AMS Classification Subjects 2020:} 60F10, 15B52

\section{Introduction}\label{chap:intro}
Hermitian random matrices have attracted  tensive attention of experts from Mathematics, Physics, Statistics and so on. Many magnificent results have been established. For example, as  weak limits of the empirical measures of the corresponding eigenvalues of Hermitian random matrices, the semi-circular law, the Marchenko-Pastur law, as well as Wachter law appeared. Also, the Tracy-Widom law and the Gumbel distribution showed up as the limits of the extremal eigenvalues. The readers are referred to \cite{ABDbook}, \cite{AnG}, \cite{BSbook}, \cite{BDSbook}, and \cite{Forrester} etc for details. 

Researchers are also interested in non-Hermitian random matrices, which have important applications ranging from the fractional Hall effect to  quantum chromodynamics. The readers may see \cite{DiF}, \cite{KSbook} and \cite{Stephanov} for necessary information and further applications. The eigenvalues of such kind matrices are complex. Researchers revealed some properties on the Ginibre ensembles, the product of truncations Haar unitary matrices, the elliptic ensembles, and the chiral non-Hermitian random matrix ensembles. 

In this parper, we will study the spectral radius of  a particular case of the chiral non-Hermitian random matrix ensembles. The spectral radius of a matrix $M$, with eigenvalues $\zeta_1, \ldots, \zeta_n$, is defined as $\max_{1\leq j \leq n}|\zeta_j|$. 
The study of edge behavior in classical ensembles of Gaussian Hermitian matrices has famously led to the discovery of Tracy-Widom distributions.
 Pioneering work by Rider \cite{Rider03} and Rider and Sinclair \cite{Rider 14} has delved into the spectral radii of the real, complex, and quaternion Ginibre ensembles. 
 They demonstrated that appropriately scaled spectral radii converge in law to a Gumbel distribution, contrasting with the Tracy-Widom distribution. 
 Further generalizations of this result can be found in the work of Chafai and P\'{e}ch\'{e} \cite{Chafaii}. 
Additionally, recent research by Lacroix-A-Chez-Toine \emph{ et al.} \cite{Lacroix} has explored the extreme statistics in the complex Ginibre ensemble, showing an intermediate fluctuation regime between large and typical fluctuations to the left of the mean. Their findings indicate that this interpolation also applies to general Coulomb gases.

For integers $n\ge 1$ and $v\ge 0,$ let $P$ and $Q$ be $(n+v)\times n$ matrices with i.i.d centered complex Gaussian entries of variance $1 /(4 \mathrm{n}).$ These are the building
blocks of the two correlated random matrices
$$
\Phi = \sqrt{1+\tau} P+\sqrt{1-\tau} Q
\text{\quad  and \quad }
\Psi = \sqrt{1+\tau} P-\sqrt{1-\tau} Q.
$$
Here, $\tau \in[0,1]$ is a non-Hermiticity parameter. Notably, for $\tau=0$, $\Phi$ and $\Psi$ are uncorrelated Gaussian random matrices, while for $\tau=1$, they become perfectly correlated, i.e., $\Phi = \Psi$. We are interested in the eigenvalue distribution of the $(2 n+\nu) \times(2 n+\nu)$ random Dirac matrix
$$\mathcal{D}=\left(\begin{matrix}{}
  0 &\Phi \\
 \Psi^*  & 0
\end{matrix} \right),
$$
where $ \Psi^*$ is the complex conjugate matrix of $\Psi$. 
This matrix was first undertaken by Osborn \cite{osborn}, in the context of a study relating to quantum chromodynamics.
The spectrum of $\mathcal{D}$ consists of an eigenvalue of multiplicity $\nu$ at the origin, and $2 n$ complex eigenvalues $\left\{ \pm \zeta_k\right\}_{k=1}^n$ which come in pairs with opposite sign. 
Without loss of generality, we choose the $\zeta_k$ to have non-negative real part. Thus the joint density function of $\zeta_1, \cdots, \zeta_{n}$ is proportional to   
$$
\prod_{1 \leq j<k \leq n}\left|z_{j}^{2}-z_{k}^{2}\right|^{2} \cdot \prod_{j=1}^{n}\left|z_{j}\right|^{2(v+1)} \exp \left(\frac{2 \tau n \operatorname{Re}\left(z_{i}^{2}\right)}{1-\tau^{2}}\right) K_{v}\left(\frac{2 n\left|z_{j}\right|^{2}}{1-\tau^{2}}\right)
$$
for all $z_1, \cdots, z_n\in \mathbb{C}$ with ${\rm Re}(z_j)>0$ for each $j,$ where $K_v$ is the modified Bessel function of the second kind (more details on $K_v$ are given in Lemma \ref{kp} below). Here and later, the density is relative to the Lebesgue  measure on $\mathbb{C}^n.$ For details on the model, we refer to \cite{osborn} with $\mu=\sqrt{(1-\tau)} / \sqrt{(1+\tau)}, \mathcal{D}_{I I}=i \mathcal{D}, \alpha=2 /(1+\tau), A=\sqrt{1+\tau} P, B=i \sqrt{1+\tau} Q, C=i \Phi$ and $D=i \Psi^*.$

In previous research, the limiting empirical distribution associated with the Dirac matrix $D$ for general $\tau \in [0,1)$ was established in \cite{ABK}. For the case $\tau = 1$, the empirical measures converge weakly to the Marckenko-Pastur law \cite{AnG, BSbook, MP}.

Let $\zeta_1, \cdots, \zeta_n$ have density function as in \eqref{density}.
Given $\tau\in [0, 1),$ and $v\ge 0$ is fixed, the authors of \cite{AB} proved that  the $\max_{1\le j\le n} {\rm Re}(\zeta_j)$ converged to the Gumbel distribution and also with different scaling, $({\rm Re}(\zeta_j), {\rm Im}(\zeta_j)),$ as a point process, converged to a Poisson  process.  For the case $\tau =1,$ the largest corresponding eigenvalue converges to the Tracy-Widom distribution  (Johansson \cite{Johansson}).

For the particular case $\tau=0,$  Chang, Jiang and Qi in \cite{JQ}  established a central limit theorem for $(n/(n+v))^{1/2}\max_{1\le j\le n}|\zeta_j|^2,$ which says that 
\begin{equation}\label{sjq0}\sqrt{\log s_n}\left(4\sqrt{s_n}\left(\left(\frac{n}{n+v} \right)^{1/4}\max_{1\le j\le n}|\zeta_j|-1\right)-a(s_n)\right)\end{equation}  converges weakly to the Gumbel distribution. 
The limit \eqref{sjq0} is equivalent to 
the following limit 
\begin{equation}\label{sjq}\sqrt{\log s_n}	\mathbb{P}\left(\left(\frac{n}{n+v}\right)^{1/2}\max_{1\le i\le n}|\zeta_i|^2\ge 1+y \right)  \sim  	1 - G\big(\sqrt{\log s_n}(2\sqrt{s_n} y -a(s_n))\big)\end{equation} 
for $y = a(s_n)/(2\sqrt{s_n}) + O ( \sqrt{\log s_n /s_n} ).$
Here, $s_n=n(n+v)/(2n+v),$ and the function $a$ is defined as $a(y)=(\log y)^{1/2}-(\log y)^{-1/2}\log(2\pi \log y)$, and $G(y)=\exp (-\exp (-y))$ is the distribution function of the Gumbel distribution. This result directly motivated our paper.

 Since $\lim_{n\to\infty}s_n=\infty,$ we observe from \eqref{sjq} that $(n/(n+v))^{1/4}\max_{1\le j\le n}|\zeta_j|$ converges to $1$ in probability. 
The main aim of this paper is to provide the deviation probabilities of $(n/(n+v))^{1/2}\max_{1\le j\le n}|\zeta_j|^2$ as in \eqref{sjq} for $y=O(1)$, as well as the moderate deviation probabilities for $(n/(n+v))^{1/2}\max_{1\le j\le n}|\zeta_j|^2$ when $\frac{a(s_n)}{s_n} \ll y \ll 1$.
 %The main aim of this paper is to offer the speed of this convergence, that is the deviation probabilities and the moderate deviation probabilities of $(n/(n+v))^{1/2}\max_{1\le j\le n}|\zeta_j|^2$. 
One may also wonder whether $n/(n+v))^{1/2}\min_{1\le j\le n}|\zeta_j|^2$ will converge to $0,$ the left endpoint of the weak limit of the empirical measure of $(n/(n+v))^{1/4}|\zeta_j|,$ and what is the speed of this convergence.

 When $\tau=0,$ the density function becomes 
\begin{equation}\label{density}
	f(z_1, \cdots, z_n)=C\prod_{1\le j<k\le n}|z_j^2-z_k^2|^2 \prod_{j=1}^n |z_j|^{2(v+1)}K_v(2n |z_j|^2).
\end{equation} 
 Here are the deviation probabilities we obtain for $\max_{1\le j\le n}|\zeta_j|^2$ and $\min_{1\le j\le n}|\zeta_j|^2,$ respectively.  
 
 \subsection{Large deviation probabilities for $\max_{1\le j\le n}|\zeta_j|^2$ and $\min_{1\le j\le n}|\zeta_j|^2$} 
 \begin{thm}\label{extremaxdp} 
  Let $(\zeta_1, \cdots, \zeta_{n})$ be random vector whose density function is given in \eqref{density}.  Suppose that $\alpha=\lim\limits_{n\to\infty} v/n\in [0, \infty].$ Then, 
  $$\lim_{n\to\infty}\frac{1}{n}\log\mathbb{P}\left(\left(\frac{n}{n+v}\right)^{1/2}\max_{1\le i\le n}|\zeta_i|^2\ge x\right)=-I^{(r)}_{\alpha}(x)$$ for any $x>1$ and 
  $$\lim_{n\to\infty}\frac{1}{n^2}\log\mathbb{P}\left(\left(\frac{n}{n+v}\right)^{1/2}\max_{1\le i\le n}|\zeta_i|^2\le x\right)=-I^{(l)}_{\alpha}(x)$$ for any $0<x<1.$
   Here, $I_{\alpha}^{(r)}$ and $I_{\alpha}^{(l)}$ are, respectively, given by 
  $$ I_{\alpha}^{(r)}(x)= \alpha\log\frac{1+\alpha}{\alpha+\kappa_{\alpha}(x)}+2(\kappa_{\alpha}(x) -\log x-1)	
$$ and 
$$ I_{\alpha}^{(l)}(x)= \left(\alpha+\frac{\alpha^2}{2}\right)\log\frac{1+\alpha}{\kappa_{\alpha}(x)+\alpha}-\log x-\frac{(\alpha+3-\kappa_{\alpha}(x))(1-\kappa_{\alpha}(x))}{2} 	
$$
and $\kappa_{\alpha}(x)=\frac{2(1+\alpha)x^2}{\alpha+\sqrt{\alpha^2+4(1+\alpha) x^2}}$ for $\alpha\in (0, +\infty).$ For $\alpha=0$ or $\alpha=+\infty,$ we take the corresponding limits, respectively. 
 \end{thm} 

 \begin{thm}\label{extremindp} 
  Let $(\zeta_1, \cdots, \zeta_n), \alpha$ and $\kappa_{\alpha}(x)$ be the same as in Theorem \ref{extremaxdp} above. Then 
  $$\lim_{n\to\infty}\frac{1}{n^2}\log\mathbb{P}\left(\left(\frac{n}{n+v}\right)^{1/2}\min_{1\le i\le n}|\zeta_i|^2\ge x\right)=-\widehat{J}_{\alpha}(x)$$ for any $x>0,$ where $\widehat{J}_{\alpha}(x)$ is defined as follows when $\alpha\in (0, +\infty)$
 $$
 \widehat{J}_{\alpha}(x)=\begin{cases} \alpha\left(\frac{\alpha+2}{2}\log\left(1+\frac1{\alpha}\right)-\log\left(1+\frac{\kappa_{\alpha}(x)}{\alpha}\right)\right) +2\kappa_{\alpha}(x)-\frac{3+\alpha}2-\log x, & x\ge 1;\\
\frac{\alpha^2}{2}\log\left(1+\frac{\kappa_{\alpha}}{\alpha}\right)-\frac{1}{2}(\alpha\kappa_{\alpha}(x)-\kappa_{\alpha}^2(x)), & 0< x<1.
 \end{cases}
  $$
  When $\alpha=0$ or $\alpha=+\infty,$ we take the corresponding limit.

 \end{thm} 

Having these two theorems at hand, we could offer a strong law of large numbers  for $(n /(n+v))^{1 / 4}\max_{1\le j\le n}|\zeta_j|,$ as well as for $(n /(n+v))^{1 / 4}\min_{1\le j\le n}|\zeta_j|.$    
\begin{rmk} Theorems \ref{extremaxdp} and the Borel Cantelli lemma lead that  $$\left(\frac{n}{n+v}\right)^{1/2}\max_{1\le j\le n}|\zeta_j|^2\stackrel{a. s. }{\longrightarrow} 1 $$ and, as a parallel result, Theorem \ref{extremindp} brings us the following strong law of large number $$\left(\frac{n}{n+v}\right)^{1/2}\min_{1\le j\le n}|\zeta_j|^2\stackrel{a. s. }{\longrightarrow} 0.$$
\end{rmk}

As we can see, the central limit \eqref{sjq0} also implies that $$\sqrt{s_n} \left(\left(\frac{n}{n+v}\right)^{1/4}\max_{1\le j\le n}|\zeta_j|-1-\frac{a(s_n)}{4\sqrt{s_n}}\right)$$ converges to $0$ in probability. This inspires us to investigate the following moderate deviation probabilities, which give more details on the behavior of $(n /(n+v))^{1 / 4}\max_{1\le j\le n}|\zeta_j|$ around $1$ and $(n /(n+v))^{1 / 4}\min_{1\le j\le n}|\zeta_j|$ around $0,$ respectively. 
 
\subsection{Moderate deviation probabilities for $\max_{1\le j\le n}|\zeta_j|^2$ and $\min_{1\le j\le n}|\zeta_j|^2$} 
For simplicity, we use the notation $x_n\ll y_n$ to replace the limit $\lim\limits_{n\to\infty}\frac{x_n}{y_n}=0.$   
We first state the moderate deviation of $\max_{1\le i\le n}|\zeta_i|^2.$
\begin{thm}\label{extremaxmdp}
	Let $(\zeta_1, \cdots, \zeta_n)$ and $\alpha$ be the same as in Theorem \ref{extremaxdp} above. Given $x>0.$ We have   
	$$\lim_{n\to\infty}\frac{1}{n \, l_n^2}\log \mathbb{P}\left(\left(\frac{n}{n+v}\right)^{1/2}\max_{1\le i\le n}|\zeta_i|^2\ge 1+ l_n \; x\right)=-\frac{2(1+\alpha)}{2+\alpha}x^2$$ for any positive sequence $(l_n)$ satisfying $\log n/n<\!< l_n^2<\!<1$ 
	and
	$$\lim_{n\to\infty}\frac{1}{n^2 \, l_n^3}\log \mathbb{P}\left(\left(\frac{n}{n+v}\right)^{1/2}\max_{1\le i\le n}|\zeta_i|^2\le 1-l_n \; x\right)=-\frac{4(1+\alpha)^2}{3(2+\alpha)^2}x^3$$
	once $\log n/n<\!< l_n^3<\!<1.$   
\end{thm}

Now, we present similar result for $\min_{1\le i\le n}|\zeta_i|^2$
\begin{thm}\label{extreminmdp} Let $(\zeta_1, \cdots, \zeta_n)$ and $\alpha$ be the same as in Theorem \ref{extremaxdp} above. Given arbitrary $x>0.$ 
\begin{enumerate} 
\item Assume that $\lim\limits_{n\to\infty}\frac{v}{\sqrt{n\log n}}=\beta\in [0, +\infty).$ Given any positive sequence $(l_{n})$ verifying $\frac{\log n}{n}<\!< l_n^2<\!<1.$ We have that 
	$$\lim_{n\to\infty}\frac{1}{n^2 \, l_{n}^2}\log \mathbb{P}\left(\left(\frac{n}{n+v}\right)^{1/2}\min_{1\le i\le n}|\zeta_i|^2\ge l_{n}\; x\right)=-\frac{x^2}{2}.$$
	\item Suppose that $\sqrt{n\log n}<\!<v<\!<n.$ It holds that 
	$$\aligned &\lim_{n\to\infty}\frac{1}{v^2}\log \mathbb{P}\left(\left(\frac{n}{n+v}\right)^{1/2}\min_{1\le i\le n}|\zeta_i|^2\ge \frac{v}{n}x \right)=-\frac{1}{2}\log\left(\frac{1+\sqrt{1+4x^2}}{2}+1+x^2-\sqrt{1+4x^2}\right); \\
	&\lim_{n\to\infty}\frac{1}{n^2 l_{n, v}^2}\log \mathbb{P}\left(\left(\frac{n}{n+v}\right)^{1/2}\min_{1\le i\le n}|\zeta_i|^2\ge l_{n, v} x\right)=-\frac{x^2}{2}, \quad v/n<\!<l_{n, v}<\!<1.
	\endaligned  $$ 
	\item Assume now $\alpha\in (0, +\infty].$  It goes 
$$\lim_{n\to\infty}\frac{1}{n^2 l_n^4}\log\mathbb{P}\left(\left(\frac{n}{n+v}\right)^{1/2} \min_{1\le i\le n}|\zeta_i|^2\ge l_{n} \, x\right)=-\frac{\left(1+\alpha\right)^2}{4\alpha^2} x^4$$ for any sequence $(l_n)$ satisfying $$\left(\frac{\log n}{n}\right)^{1/4} <\!<l_{n}<\!< 1.$$
\end{enumerate}
\end{thm}
\begin{rmk}
The deviation probabilities are a specific case of the hole probability, which refers to the probability that a prescribed region is free of eigenvalues. This concept was initially explored in the context of the complex Ginibre ensemble in \cite{Grobe}. The generalized form, known as the Mittag-Leffler ensemble, has been the subject of numerous recent studies, as documented in recent paper \cite{Charlier03, Charlier22}.
\end{rmk}

%\begin{rmk}
%In \cite{LACT}, the extreme statistical properties of the Ginibre ensemble were investigated, focusing on the interpolation between the Gumbel law and the large deviation regime. The findings indicate that the same interpolation applies to more general potentials. As a particular case, this paper yields similar results, showing that the rate function of moderate deviation can smoothly match both the Gumbel distribution and the large deviation rate function. 
%\end{rmk}
The remainder of this paper will be organized as follows : in the following section, we give some necessary lemmas and the third section is devoted to the proofs related to the case that $v$ is bounded. In this last section, we give the proofs corresponding to the case that $v\to\infty.$  

\section{Preliminaries} 
 In this section, we collect some key lemmas.   
 
 The first crucial one is borrowed from \cite{JQ}, originally coming from \cite{Kostlan}, which transfers the corresponding probabilities of 
 $(n /(n+v))^{1/2}\min_{1\le i\le n}|\zeta_i|^2$ and $(n /(n+v))^{1/2}\max_{1\le i\le n}|\zeta_i|^2$ to those related to independently random variables taking values in real line. It says that 
 \begin{lem}\label{s}
	 Let $\zeta_1, \cdots, \zeta_n$ be random variables with a joint density function given by \eqref{density}. Let $(Y_j)_{1\le j\le n}$ be independent random variables and the density function of $Y_j$ is proportional to $y^{2j+v-1} K_v(2n y) 1_{y>0}.$ Then for any symmetric function $g,$ $g(|\zeta_1^2|, \cdots, |\zeta_n^2|)$ and $g(Y_1, \cdots, Y_n)$ have the same distribution. 
\end{lem}
 
The modified Bessel function of the second kind $K_v$ plays crucial role in the density function of $(Y_j)_{1\le j\le n}.$ We first introduce some properties of $K_v$ in the following lemma  
(see \cite{AS} and \cite{Olver}). 
\begin{lem}\label{kp} The following statements hold:
\begin{itemize} 
\item[(1).] (Formula 10.32.8 in \cite{Olver}) Given $v\ge 0.$ It holds that 
$$K_v(x)=\frac{\sqrt{\pi} x^v}{2^v \Gamma(v+1/2)}\int_1^{\infty} e^{-x t}(t^2-1)^{v-1/2} dt$$ for any $x> 0.$
\item[(2).] (Formula 10.40.2 in \cite{Olver} ) Suppose that $v\ge 0$ is fixed. Then, $$K_{v}(x)=\sqrt{\frac{\pi}{2 x}} e^{-x}\left[1+O\left(\frac{1}{x}\right)\right] \quad \text { as } x \rightarrow \infty.$$
\item[(3).] (Formula 10.41.4 in \cite{Olver}) As $v\to\infty,$ we have that 
$$K_{v}(v x)=\sqrt{\frac{\pi}{2 v}} x^{-v}\left(1+x^{2}\right)^{-1/4} e^{-v (\sqrt{1+x^2}-\log(1+\sqrt{1+x^2})}\left[1+O\left(\frac{1}{v}\right)\right]$$ 
uniformly on $x\in (0, +\infty).$ 
\end{itemize} 	
\end{lem}

For the deviation probabilities of $(n /(n+v))^{1/2}\min_{1\le i\le n}|\zeta^2_i|$ and $(n /(n+v))^{1/2} \max_{1\le i\le n}|\zeta^2_i|,$ starting from the observation of Lemma \ref{kp}, we need treat the following two kinds integrals  \\
$$\int_{a}^{b} y^{2j+v-1} K_v(y) dy, \quad \int_{a}^b y^{2j+v-1} K_v(vy) dy$$ for $v$ bounded and  
 $v$ sufficiently large, respectively. Thus, we offer some estimates on these two integrals and the normalizing constant 
 \begin{equation}\label{zjconstant} Z_j:=\int_{0}^{\infty} y^{2j+v-1} K_v(y) dy.\end{equation}

\begin{lem}\label{constant}
Let $Z_j$ be defined as in \eqref{zjconstant}. Then, for any $j\ge 1$ and $v\ge 0,$ we have that 
\begin{equation}\label{zjnew} Z_j=\frac{\sqrt{\pi}\,\Gamma(2j+2v)\,\Gamma(j)}{2^{v+1}\Gamma(j+v+1/2)}.\end{equation}	
Moreover,   
$$\log Z_j=2j\log (2j)-2j+O(\log j)$$ for $j$ large enough and $v$ bounded.  
\end{lem}

The expression \eqref{zjnew} follows from the planar orthogonality and its norm, see e.g. \cite{Akemann} and the asymptotic is due to the Stirling formula $$\Gamma(z)=z^{z-1/2} e^{-z} \sqrt{2\pi}(1+o(1))$$ for $z$ large enough implies the asymptotic.  
  
Seeing the behavior $K_v(y)=C e^{-y}(1+O(y^{-1}))$ as $y\to\infty,$ we first offer lower and upper bounds on the corresponding integral of $e^{-y} y^{b}$ for $b>0.$  
\begin{lem}\label{estioninte} Given any $b>0$ fixed.  
\begin{enumerate}
\item For any $a\ge b+1,$ 
$$a^b e^{-a}\le \int_a^{\infty} y^b e^{-y} dy\le a^{b+1} e^{-a}. $$
\item For any $a<b+1,$ 
$$b^b e^{-(b+1)}\le \int_a^{\infty} y^b e^{-y} dy\le 2(b+1)b^{b} e^{-b}. $$
\item For any $a>b,$ 
$$\frac{1}{b+1} b^{b+1} e^{-b}\le \int_0^a y^b e^{-y} dy\le a \;b^{b} e^{-b}. $$
\item For any $a<b-1,$ 
 $$\frac{1}{b+1}{a}^{b+1} e^{-a}\le \int_0^a y^b e^{-y} dy\le a^{b+1} e^{-a}. $$
\end{enumerate}   
\end{lem}
\begin{proof} Set a function $w$ as 
$$w(y)=y-b\log y.$$ As we can see, $w$ is a strictly convex function, whose unique minimizer is nothing but $b$ and $w'$ is increasing on $[b, +\infty)$ while is decreasing on $(0, b].$  
Suppose now that $a>b.$ Therefore, $w'$ is positive and increasing on $[a, +\infty),$ which implies  that
$$\aligned \int_a^{+\infty} e^{-w(y)}dy=\int_a^{+\infty} e^{-w(y)} \frac{w'(y)}{w'(y)} dy\le \frac{1}{w'(a)} \int_{a}^{\infty}d(-e^{-w(y)})=\frac{a^{b+1} e^{-a}}{a-b}. 
\endaligned $$
Here we use the fact that  $w'(a)=1-\frac{b}{a}.$
A similar reason allows us to get the following lower bound  
$$\aligned \int_a^{\infty} e^{-w(y)}dy\ge \frac{1}{w'(\infty)}e^{-w(a)}=a^b e^{-a}. 
\endaligned $$
For the second inequality, we 
see 
$$\aligned \int_a^{+\infty} e^{-w(y)}dy&=\int_a^{b+1} e^{-w(y)} dy+\int_{b+1}^{+\infty} e^{-w(y)} dy\\
&\le (b+1-a) e^{-w(b)}+(b+1) e^{-w(b+1)}\\
&\le 2(b+1) e^{-w(b)}=2(b+1)b^b e^{-b}. 
\endaligned $$ 
The last inequality is true since $b$ is the minimizer of $w.$  
It is easy to see that 
$$\int_a^{+\infty} y^b e^{-y} dy\ge \int_{b+1}^{\infty} y^b e^{-y}dy\ge b^b e^{-(b+1)}.$$ 
For $a>b,$ the upper bound in the third item is natural since $y^b e^{-y}\le b^b e^{-b}.$ The lower bound shows up because 
$$\int_0^a y^b e^{-y} dy\ge \int_0^b y^{b} e^{-y}dy\ge e^{-b}\int_0^b y^{b} dy=\frac{1}{b+1} b^{b+1} e^{-b}.$$ The proof of the last item is similar, and is omitted here. 
\end{proof} 
 Next, we are going to work on the integral $\int_a^b y^{2j+v-1} K_{v}(vy) dy$ for $v$ large enough. Guided by the third term of Lemma \ref{kp}, for simplicity of the notations, we introduce the following two functions 
$$ u(x):=\sqrt{1+x^2}-\log(1+\sqrt{1+x^2}),$$ and  
$$\tau_j(x)=u(x)+\frac{1}{4v}\log(1+x^2)-\frac{2j-1}{v}\log x 
$$ for $x>0$ and $1\le j\le n.$
Utilize the third item of Lemma \ref{kp} to write   $$ y^{2j+v-1} K_v(vy)=\sqrt{\frac{\pi}{2v}} e^{-v\tau_j(y)}(1+O(v^{-1}))$$ uniformly on $y\in (0, +\infty).$  
Some simple calculus brings the derivative and second derivative of $\tau_j$ as 
$$\tau_j'(y)=\frac{y}{1+\sqrt{1+y^2}}+\frac{y}{2v(1+y^2)}-\frac{2j-1}{v \, y}$$ and 
$$\tau_j''(y)=\frac{1}{\sqrt{1+y^2}(1+\sqrt{1+y^2})}+\frac{1-y^2}{2 v(1+y^2)^2}+\frac{2j-1}{v y^2}.$$ 
We see clearly that $\tau_j'(0)=-\infty$ and $\tau_j(+\infty)=1.$   
Also, it is true that $\tau_j''>0,$ since 
$$\frac{1-y^2}{2 v(1+y^2)^2}+\frac{2j-1}{v y^2}=\frac{y^2(1-y^2)+(2j-1)(1+y^2)^2}{2v y^2(1+y^2)^2}\ge\frac{3y^2+1}{2v y^2(1+y^2)^2}.$$ 
Therefore, there exists a unique solution to the equation $\tau_j'=0,$ which is obviously the unique minimizer of $\tau_j$ and is denoted by $x_j.$ Next, we give a lemma on $x_j.$  
\begin{lem}\label{keyl} Let $\tau_j$ and $x_j$ be defined as above.  
\begin{enumerate}  
\item For any $1\le j\le n,$ 
$$2\sqrt{\left(j-\frac34\right)\left(j+v-\frac34\right)}< v \, x_j< 2\sqrt{\left(j-\frac12\right)\left(j+v-\frac12\right)}.$$ 
\item For any $a>x_j,$ it holds that
$$\frac{1}{v \tau_j'(M)}e^{-v\tau_j(a)}\left(1-e^{v (\tau_j(a)-\tau_j(M))}\right)\le \int_a^{\infty} e^{-v \tau_j(t)}dt\le \frac1{v  \tau_j'(a)}e^{-v \tau_j(a)} $$
for any $M>a.$ 
\item For any $a<x_j,$ we know that 
$$\aligned \int_a^{\infty} e^{-v \,\tau_j(t)}dt&\le 4 j e^{-v\,\tau_j(x_j)},\\
\int_a^{\infty} e^{-v \,\tau_j(t)}dt&\ge \frac{1}{v\,\tau_j'(M)} e^{-v\,\tau_j(x_j)}\left(1-e^{-v(\tau_j(M)-\tau_j(x_j))}\right)\endaligned $$
for any $M>x_j.$ 
\item Given $a<x_j.$ One gets that  
$$\aligned \frac{1}{-v\tau_j'(a_1)}e^{-v\tau_j(a)}\left(1-e^{v\tau_j(a)-v\tau_j(a_1)}\right)\le\int_0^a e^{-v \,\tau_j(t)}dt&\le \frac{1}{-v\tau_j'(a)}e^{-v\,\tau_j(a)} \\
\endaligned $$
for any $a_1\in (0, a).$ 
\item It holds that 
$$\frac{1}{-v\tau_j'(M)}e^{-v\tau_j(x_j)}\left(1-e^{v\tau_j(x_j)-v\tau_j(M)}\right)\le\int_0^a e^{-v \,\tau_j(t)}dt\le a e^{-v\,\tau_j(x_j)}\\
$$
for any $a>x_j>M>0.$
\end{enumerate} 
\end{lem}
\begin{proof}  
Set $y_j(t)=2\sqrt{(j-t)(j+v-t)}/v,$ which satisfies the following equation 
$$1+y_j^2(t)=(2j-2t+v)^2/v^2.$$ 
Hence, 
$$y_j(t)\tau'_j(y_j(t))=\frac{(2-4t)v^2+(3-4t)(2j-2t)(2j-2t+2v)}{2v(2j-2t+v)^2}.$$  As a consequence, we know that the numerator is positive when $t=1/2,$ and is negative when $t=3/4$ instead. 
Since $\tau'_j(x_j)=0$ and 	$\tau'_j$ is strictly increasing, the second item is proved. 
By the property of $\tau_j,$ we know that both $\tau_j$ and $\tau_j'$ are increasing on 
$[x_j, +\infty).$ Hence, for $a>x_j,$
$$\aligned \int_a^{+\infty} e^{-v\tau_j(t)}dt&=\int_a^{+\infty} e^{-v\tau_j(t)} \frac{v  \tau_j'(t)}{v \tau_j'(t)}dt\\
 &\le \frac{1}{v\inf_{x\ge a} \tau_j'(t)} \int_a^{\infty} d\left(-e^{-v \;\tau_j(t)}\right) \\
 &=\frac{1}{v\, \tau_j'(a)} e^{-v \,\tau_j(a)}.
\endaligned $$ 
Also, for any $M>a,$ 
$$\aligned \int_a^{+\infty} e^{-v\tau_j(t)}dt&\ge\int_a^{M} e^{-v\tau_j(t)} \frac{v  \tau_j'(t)}{v \tau_j'(t)}dt\\
 &\ge\frac{1}{v\sup_{x\in(a, M)} \tau_j'(t)} \int_a^{M} d\left(-e^{-v \tau_j(t)}\right) \\
 &=\frac{1}{v\, \tau_j'(M)} e^{-v \,\tau_j(a)}\left(1-e^{-v\tau_j(M)+v\tau_j(a)}\right).
\endaligned $$
Meanwhile, similarly for $a<x_j,$ we have that
$$\aligned \int_a^{+\infty} e^{-v\tau_j(t)}dt&=\int_a^{2\sqrt{j(j+v)}/v} e^{-v\tau_j(t)}dt+\int_{2\sqrt{j(j+v)}/v}^{+\infty} e^{-v\tau_j(t)}dt\\
&\le e^{-v\tau_j(x_j)}\left(\frac{2\sqrt{j(j+v)}}{v}+\frac{1}{v\tau_j'(2\sqrt{j(j+v)}/v)}\right).
\endaligned $$  
Here, the last inequality holds since $x_j$ is the maximizer of $-\tau_j$ and $2\sqrt{j(j+v)}/v$ is an upper bound of $x_j$ by the second item. The definition of $\tau_j$ and some simple calculus  lead
$$\aligned \tau_j'(2\sqrt{j(j+v)}/v)&=\sqrt{j/v}+\sqrt{j(j+v)}/(v+2j)^2-(2j-1)/(2\sqrt{j(j+v)})\\
&\ge 1/(2\sqrt{j(j+v)}). 
\endaligned $$
Clearly, 
$$\frac{2\sqrt{j(j+v)}}{v}+\frac{1}{v\tau_j'(2\sqrt{j(j+v)}/v)}\le \frac{4\sqrt{j(j+v)}}{v}\le 4j.$$
The second inequality in the fourth item follows from a similar reason to the lower bound in the third item, and the proof is omitted here. The proof is then complete. 
\end{proof} 

At last, we state a lemma on a particular sum, which will be frequently used later.  
\begin{lem}\label{ma} 
For $n$ large enough, it holds that
$$\aligned 
&\sum_{i=1}^n i \log i=-\frac{n^2}{4}+\frac{n(n+1)}{2}\log n+\frac1{12}\log n+O(1); \\ 
&\sum_{i=1}^n \left(i+v-\frac 1 2\right)\log (i+v)=-\frac{n^2+2nv-2n}{4}+\frac{(n+v)^2}{2}\log(n+v)\\
&\quad \quad \quad \quad\quad \quad\quad \quad\quad \quad\quad \quad\;-\frac{v^2}{2}\log v-\frac16\log\left(1+ \frac n v\right) +O(1).
\endaligned $$ 	
\end{lem}
\begin{proof} It was proved in \cite{MaLDP} that 
$$\sum_{i=1}^n i\log \frac{i}{n}=-\frac{n^2}{4}+\frac{1}{12}\log n+O(1), \quad \sum_{i=1}^n \log\frac{i}{n}=-n+\frac12\log n+O(1).$$ 
The first item follows if we write $$\sum_{i=1}^n i \log i=\sum_{i=1}^n i \log \frac i n+\frac{n(n+1)}{2}\log n, \quad \sum_{i=1}^n \log i=\sum_{i=1}^n \log \frac i n+n\log n.$$ 
The second item is true, since  
$$\aligned &\quad \sum_{i=1}^n (i+v-1/2)\log (i+v)\\
&=\sum_{i=1}^{n+v} (i-1/2)\log i-\sum_{i=1}^v (i-1/2)\log i\\
&=-\frac{(n+v)^2-v^2}{4}+\frac{n}{2}+\frac{(n+v)^2}{2}\log(n+v)-\frac{v^2}{2}\log v-\frac16\log\left(1+\frac nv \right)+O(1).
\endaligned $$   	
\end{proof}

\section{When $v$ is boudned}  
In this section, we work on the general assumption that $v$ is bounded, under which we give some asymptotic of $\log \mathbb{P}(\sqrt{n/(n+v)} Y_j\ge a)$ and $\log \mathbb{P}(\sqrt{n/(n+v)} Y_j\le a)$ for $a>0$ based on Lemmas \ref{constant} and \ref{estioninte}, respectively.  
For simplicity, set 
\begin{align}\label{defx}
	X_j=\left(\frac{n}{n+v}\right)^{1/2} Y_j, \quad X_{(n)}:=\max_{1\le j\le n} X_j \quad \text{and} \quad X_{(1)}:=\min_{1\le j\le n} X_j
\end{align}
and $c_{n, v}=2\sqrt{n(n+v)}.$   Thus, 
 $$\log \mathbb{P}(X_j\ge a)=\log \mathbb{P}(2n Y_j\ge c_{n, v} \,a)$$ 
 and 
 $$\log \mathbb{P}(X_j\le a)=\log \mathbb{P}(2n Y_j\le c_{n, v} \,a).$$
  Having these facts at hand, we offer two key lemmas as follows.
 
 \begin{lem}\label{asymlem1} Consider the random variables $(Y_j)$ and given $a_{n}>0$ such that $c_{n, v} \,a_{n}\to\infty$ and $c_{n, v} a_{n}/n$ is bounded.   
 \begin{enumerate}
\item When $j$ satisfy $2j+v-1/2>c_{n, v} \, a_{n},$ we have the following asymptotic
\begin{equation}\label{asymb1}\log\mathbb{P}(2n Y_j\ge c_{n, v} a_{n})=\tilde{O}(\log n).\end{equation}
\item By contrast, for $j$ such that  $2j+v-1/2\le c_{n, v} \, a_{n},$	we have that
\begin{equation}\label{asymb2}\log\mathbb{P}(2n Y_j\ge c_{n, v} a_{n})=-2j\log \frac{j}{n a_{n}}+2j-2n a_{n}+\tilde{O}(\log n).\end{equation} 
 \end{enumerate}	
  Here, $s_n=\tilde{O}(\log n)$ means $\varlimsup\limits_{n\to\infty}\frac{s_n}{\log n}<\infty.$
 \end{lem}
\begin{proof} By definition, 
$$\mathbb{P}(2n Y_j\ge c_{n, v} a_{n})=Z_j^{-1}\int_{c_{n, v} a_{n}}^{+\infty} y^{2j+v-1} K_v(y) dy,$$ with the normalizing constant $Z_j=\int_0^{\infty} y^{2j+v-1}K_v(y) dy.$ 
Since $c_{n, v}\, a_{n}$ is large enough, Lemma \ref{kp} helps us to write 
$$\aligned \mathbb{P}(2n Y_j\ge c_{n, v} a_{n})&=\sqrt{\frac{\pi}{2}}Z_j^{-1}\int_{c_{n, v} a_{n}}^{+\infty} y^{2j+v-3/2} e^{-y}dy(1+o(1)).\\
\endaligned $$ 	
For $j$ such that $2j+v-1/2>c_{n, v} a_{n},$ Lemma \ref{estioninte} tells us that 
$$\aligned\log \int_{c_{n, v} a_{n}}^{+\infty} y^{2j+v-3/2} e^{-y}dy&\le (2j+v-3/2)\log(2j+v-3/2)-2j+O(\log (c_{n, v} a_{n}))\\
&=2j\log(2j)-2j+\tilde{O}(\log n).
\endaligned $$
Also, 
$$\aligned\log \int_{c_{n, v} a_{n}}^{+\infty} y^{2j+v-3/2} e^{-y}dy&\ge (2j+v-3/2)\log(2j+v-3/2)-(2j+v-3/2)\\
&=2j\log(2j)-2j+\tilde{O}(\log n).
\endaligned $$
Lemma \ref{constant} says 
$$\log Z_j=2j\log (2j)-2j+\tilde{O}(\log n).$$ 
Taking all these facts into account, we get 
$$\log \mathbb{P}\left( 2n Y_j\ge c_{n, v} \, a_{n}\right)=\tilde{O}(\log n).$$  
Meanwhile, for $j$ satisfying $2j+v-1/2\le c_{n, v} \, a_{n, v},$ 
we see 
$$\aligned\log \int_{c_{n, v} a_{n}}^{+\infty} y^{2j+v-3/2} e^{-y}dy&\le (2j+v-1/2)\log(c_{n, v}\, a_{n})-c_{n, v}\, a_{n}\\
&=2j\log(2n a_{n})-2n a_{n}+\tilde{O}(\log n)
\endaligned $$ 
and 
$$\aligned\log \int_{c_{n, v} a_{n}}^{+\infty} y^{2j+v-3/2} e^{-y}dy&\ge (2j+v-3/2)\log(c_{n, v}\, a_{n})-c_{n, v}\, a_{n}\\
&=2j\log(2n a_{n})-2n a_{n}+\tilde{O}(\log n).
\endaligned $$
Therefore, one gets that
$$\log \mathbb{P}\left( 2n Y_j\ge c_{n, v} \, a_{n}\right)=-2j\log\frac{j}{n a_{n}}+2 j-2n a_{n}+\tilde{O}(\log n).$$
\end{proof}

The integral $\int_0^{c_{n, v} a} $ is involved in the probability $\mathbb{P}(2n Y_j\le a).$ Unlike the integral $\int_{c_{n, v} a}^{\infty},$ we can not use directly the asymptotical expression of $K_v(y)$ since it is valid only for $y$ large enough. Hence, we give a lemma on the expression of $\mathbb{P}(2n Y_j\le a).$ 

\begin{lem}\label{leexpre} Let $(Y_j)$ be as above. We claim that 
$$\label{lefact}
\mathbb{P}(2n Y_j\le c_{n, v} a)={\Gamma(2j+v-1/2)}^{-1}\int_0^{c_{n, v} a} y^{2j+v-3/2} e^{-y} dy(1+o(1)).
$$
for $a>0$ and $j$ large enough.	
\end{lem}
 \begin{proof}
 In fact, Lemma \ref{constant} and the Stirling formula work together to ensure that 
$$\aligned \log Z_j&=\log \pi+(2j+2v-1/2)\log(2j+2v)+(j-1/2)\log j\\ 
&\quad -(v+1/2)\log 2-(j+v)\log(j+v+1/2)-(2j+v-1/2)+o(1)\\
&=\log\pi+(2j+v-1)\log(2j+2v)-(j-1/2)\log(1+v/j)\\
&\quad-(j+v)\log(1+1/(2(j+v)))-(2j+v-1/2)+o(1)
\endaligned $$ 
for $j$ large enough. 
Apply the Taylor formula $\log(1+x)=x+o(x)$ for $x,$ with $|x|$ small enough, to obtain 
$$(j-1/2)\log(1+v/j)+(j+v)\log(1+1/(2(j+v)))=v+1/2+o(1)$$
and $$(2j+v-1)\log(2j+2v)=(2j+v-1)\log(2j+v-1/2)+v+1/2+o(1).$$ 
This implies 
 $$\aligned \log Z_j&=\log\pi+(2j+v-1)\log(2j+v-1/2)-(2j+v-1/2)+o(1),
\endaligned $$ 
which is equivalent to saying that
$$Z_j=\sqrt{\frac{\pi}{2}}\Gamma(2j+v-1/2)(1+o(1)).$$
Thus, 
$$\aligned \mathbb{P}(2n Y_j\le c_{n, v} a)&=1-\sqrt{\frac{\pi}{2}}Z_j^{-1} \int_{c_{n, v} a}^{+\infty} y^{2j+v-3/2} e^{-y} dy(1+o(1))\\
&=1-\frac{1}{\Gamma(2j+v-1/2)}\int_{c_{n, v} a}^{+\infty} y^{2j+v-3/2} e^{-y} dy(1+o(1))\\
&={\Gamma(2j+v-1/2)}^{-1}\int_0^{c_{n, v} a} y^{2j+v-3/2} e^{-y} dy(1+o(1)). 
\endaligned $$
The proof is complete. 
\end{proof}
Based on Lemma \ref{leexpre}, we are able to state a parallel asymptotic for 
 $\log\mathbb{P}(2n Y_j\le c_{n, v} a_{n}).$   
\begin{lem}\label{lelower1} Consider the random variables $(Y_j)$ and given $a_{n}>0$ such that $c_{n, v} \,a_{n}$ tends to the infinity and $c_{n, v} a_{n}/n$ is bounded.   
 \begin{enumerate}
\item When $j$ satisfy $2j+v-5/2>c_{n, v} \, a_{n},$ it holds 
\begin{equation}\label{asymb1}\log\mathbb{P}(2n Y_j\le c_{n, v} a_{n})=-2j\log \frac{j}{n a_{n}}+2j-2n a_{n}+\tilde{O}(\log n).\end{equation}
\item By contrast, for $j$ such that  $2j+v-5/2\le c_{n, v} \, a_{n},$	we have 
\begin{equation}\label{asymb2}\log\mathbb{P}(2n Y_j\le c_{n, v} a_{n})=\tilde{O}(\log n).\end{equation} 
 \end{enumerate}	

 \end{lem}
\begin{proof} Since $c_{n, v}\, a_{n}$ is large enough, Lemma \ref{leexpre} entails that 
$$\mathbb{P}(2n Y_j\le c_{n, v} a_{n})={\Gamma(2j+v-1/2)}^{-1}\int_0^{c_{n, v} a_{n}} y^{2j+v-3/2} e^{-y} dy(1+o(1)).$$ 
For $j$ such that $2j+v-5/2>c_{n, v} a_{n}$ Lemma \ref{estioninte} ensures the following estimates  
$$\aligned\log \int_0^{c_{n, v} a_{n}} y^{2j+v-3/2} e^{-y}dy
\le & (2j+v-1/2)\log(c_{n, v} a_{n})-c_{n, v} a_{n}\\
=&2j\log(c_{n, v} a_{n})-c_{n, v} a_{n}+O(\log c_{n, v} a_{n}); \\
\log \int_0^{c_{n, v} a_{n}} y^{2j+v-3/2} e^{-y}dy
\ge & (2j+v-1/2)\log(c_{n, v} a_{n})-c_{n, v} a_{n}\\ 
& - \log(2j+v-1/2) \\
=&2j\log(c_{n, v} a_{n})-c_{n, v} a_{n}+O(\log j).
\endaligned $$
The Stirling formula entails again that
$$\log\Gamma(2j+v-1/2)=2j\log(2j)-2j+\tilde{O}(\log n)$$ since $j$ is sufficiently large. 
Putting all these facts together, we have  
$$\log \mathbb{P}\left( 2j Y_j\le c_{n, v} \, a_{n}\right)=-2j\log \frac{j}{n a_{n}}+2j-2n a_{n}+\tilde{O}(\log n).$$  
Meanwhile, for $j$ satisfying $2j+v-5/2\le c_{n, v} \, a_{n},$ 
we see 
$$\aligned\log \int_0^{c_{n, v} a_{n}} y^{2j+v-3/2} e^{-y}dy&\le (2j+v-3/2)\log(2j+v-3/2)-2j+O\left(\log (c_{n, v} a_{n})\right)\\
&=2j\log(2j)-2j+\tilde{O}(\log n)
\endaligned $$ 
and 
$$\aligned\log \int_0^{c_{n, v} a_{n}} y^{2j+v-3/2} e^{-y}dy&\ge (2j+v-1/2)\log(2j+v-3/2)-(2j+v-3/2)+(\log j)\\
&=2j\log(2j)-2j+\tilde{O}(\log n).
\endaligned $$
Therefore, one gets that
$$\log \mathbb{P}\left( 2n Y_j\le c_{n, v} \, a_{n}\right)=\tilde{O}(\log n).$$
\end{proof}

\subsection{Proof of Theorem \ref{extremaxdp} when $v$ is bounded}  
With the help of Lemma \ref{s}, for the right deviation probability in Theorem \ref{extremaxdp}, it suffices to prove 
\begin{equation}\label{u1} 
\lim_{n\to\infty}\frac{1}{n}\log\mathbb{P}(X_{(n)}\ge x)=-I_{0}^{(r)}(x)	
\end{equation}
 for any $x>1.$
It was proved in \cite{JQ} that $\mathbb{P}(X_i\ge x)$ is increasing on $i,$ which implies  
 $$\mathbb{P}(X_{n}\ge x)\le \mathbb{P}(X_{(n)}\ge x)\le \sum_{i=1}^n \mathbb{P}(X_i\ge x)\le n\mathbb{P}(X_n\ge x).$$ 
Therefore, 
$$\lim_{n\to\infty} \frac1n\log\mathbb{P}(X_{(n)}\ge x)= \lim_{n\to\infty} \frac1n\log \mathbb{P}(X_n\ge x).$$ 
Since $c_{n, v} x\ge 2n+v-1/2$ when $x>1,$ Lemma \ref{asymlem1}, with $a_{n}\equiv x,$ guarantees that   
$$\log \mathbb{P}(X_{n}\ge x)=\log\mathbb{P}(2nY_n\ge c_{n, v} x)=2n\log x+2n-2n x+\tilde{O}(\log n).$$
Consequently, 
$$\lim_{n\to\infty}\frac{1}{n}\log \mathbb{P}(X_{n}\ge x)=2\log x+2-2x=-I_0^{(r)}(x).$$ 
Now, we prove the left deviation probability.  
Given any $0<x<1,$ it follows from the definition that 
$$\mathbb{P}(X_{(n)}\le x)=\prod_{j=1}^n \mathbb{P}(X_j\le x).$$  
Inspired by Lemma \ref{lelower1}, we set 
$$j_0:=\max\{j: 2j+v-5/2\le c_{n, v} x\}.$$ 
Lemma \ref{lelower1} enables us to claim the following two equalities 
$$ \log\mathbb{P}(X_j\le x)=\begin{cases}-2j\log \frac{j}{n x}+2j-2n x+\tilde{O}(\log n), & j\ge j_0+1;\\ 
\tilde{O}(\log n), & j\le j_0. 
\end{cases}
$$
On the one hand, with the help of Lemma \ref{ma}, we get that  
$$\aligned \sum_{j=j_0+1}^n \log\mathbb{P}(X_j\le x)&=\sum_{j=j_0+1}^n \left(2j\log(n\, x)-2n\, x-2j\log j+2j\right)+o(n^2)\\
&=(n^2-j_0^2)\log (n x)+3(n^2-j_0^2)/2-2n x(n-j_0)\\
&\quad -n^2\log n+j_0^2\log j_0+o(n^2)\\
&=3(n^2-j_0^2)/2-2 n(n-j_0)x+n^2\log x+j_0^2\log \frac{j_0}{n x}+o(n^2).
\endaligned$$  
On the other hand, we also have 
$$\aligned \sum_{j=1}^{j_0+1} \log\mathbb{P}(X_j\le x)=o(n^2).
\endaligned$$
Finally, the choice of $j_0$ ensures $j_0/n\to x$ as $n\to\infty,$ which makes sure that  
$$\lim_{n\to\infty}\frac{1}{n^2}\sum_{j=1}^n\log \mathbb{P}(X_j\le x)=\log x+\frac {x^2-4x+3}{2}=-I_0^{(l)}(x).$$
The proof of Theorem \ref{extremaxdp} for $v$ bounded is then finished.  
\subsection{Proof of Theorem \ref{extremindp} when $v$ is bounded} 
From the definition in equation \eqref{defx}, 
$$\mathbb{P}(X_{(1)}\ge x)=\prod_{1\le j\le n}\mathbb{P}(X_j\ge x)=\prod_{1\le j\le n} \mathbb{P}(2n Y_{j}\ge c_{n, v} x).$$ 
When $x\ge 1,$ $2j+v-1/2\le c_{n, v} x.$ Thus, it follows from Lemma \ref{asymlem1}, with $a_{n}\equiv x,$ that   
$$\aligned \log \mathbb{P}(X_j\ge x)&=-2j\log\frac{j}{n x} +2j-2n x+O(\log n). 
\endaligned $$
Therefore, use Lemma \ref{ma} to arrive at  
$$\aligned \sum_{j=1}^n \log\mathbb{P}(X_j\ge x)&= \sum_{j=1}^n \left(2j\log(n x)-2j\log (j)+2j-2n x\right)+o(n^2)\\
&=n^2\log (n x)+\frac{3}{2}n^2-n^2\log n-2n^2 x+o(n^2)\\
&=n^2\log x+\frac{3}{2}n^2-2n^2 x+o(n^2).
\endaligned $$
As a consequence, it holds that  
$$\lim_{n\to\infty}\frac1{n^2}\log\mathbb{P}(X_{(1)}\ge x)=-2x+\log x+\frac32=- \widehat{J}_0(x).$$   
Now we suppose $0<x<1.$ According to the condition of Lemma \ref{asymlem1}, set 
$$j_1=\max\{j: 2j+v-1/2\le c_{n, v} x\}.$$  
Write 
\begin{equation}\label{sumall}\aligned \log \mathbb{P}(X_{(1)}\ge x)=\sum_{j=1}^n \log \mathbb{P}(X_j\ge x)=\sum_{j=1}^{j_1} \log \mathbb{P}(X_j\ge x)+\sum_{j=j_1+1}^n \log \mathbb{P}(X_j\ge x).
\endaligned\end{equation}
It follows from Lemma \ref{asymlem1} that  
\begin{equation}\label{sum1} \sum_{j=j_1+1}^n\log \mathbb{P}(X_j\ge x)=O((n-j_1)\log n)=o(n^2). 
\end{equation}
Furthermore, we have that
\begin{equation}\label{sum2}\aligned \sum_{j=1}^{j_1}\log \mathbb{P}(X_j\ge x)&=\sum_{j=1}^{j_1}\left(2j\log(n x)-2j\log (j)+2j-2n x\right)+o(n^2)\\
&=j_1^2\log\frac{n x}{j_1}+\frac{3}{2}j_1^2-2n xj_1+o(n^2)\\
&=-\frac{n^2 x^2}{2}+o(n^2).
\endaligned \end{equation}
Here, the second equality is due to Lemma \ref{ma} and the last equality follows from the choice of $j_1,$ which ensures that $j_1=n x+O(1).$ 
Putting these two asymptotical expressions \eqref{sum1} and \eqref{sum2} into \eqref{sumall}, one gets the desired limit 
$$ \lim_{n\to \infty}\frac{1}{n^2}\log\mathbb{P}(X_{(1)}\ge x)=-\frac{x^2}{2}=- \widehat{J}_0(x).$$ 

\subsection{Proof of Theorem \ref{extremaxmdp} when $v$ is bounded} 
As for Theorem \ref{extremaxdp}, the condition $nl_n^2>\!>\log n$ guarantees that 
$$\lim_{n\to\infty}\frac1{nl_n^2}\log \mathbb{P}(X_{(n)}\ge 1+l_n \,x)=\lim_{n\to\infty}\frac1{nl_n^2}\log \mathbb{P}(2n Y_{n}\ge c_{n, v}(1+l_n \,x)).$$
Since $c_{n, v}(1+l_n\, x)\ge 2n+v-1/2$ for any $x>0,$ we get from Lemma \ref{asymlem1} and the assumption $n\,l_n^2>\!>\log n$ that 
$$\aligned \log\mathbb{P}(2n Y_n\ge c_{n, v}(1+l_n \, x))&=2n\log(1+l_n x)+2n-2n(1+l_n x)+\tilde{O}(\log n) \\
&=-n \,l_n^2\, x^2+o(n \,l_n^2). \endaligned $$
Thus, 
$$\lim_{n\to\infty}\frac{1}{nl_n^2} \log\mathbb{P}(X_{(n)}\ge 1+l_n x)=-x^2.$$ 
Similarly, we have that
$$\log\mathbb{P}(X_{(n)}\le 1-l_n x)=\sum_{j=1}^{j_2}\log\mathbb{P}(X_{j}\le 1-l_n x)+\sum_{j=j_2+1}^n \log\mathbb{P}(X_{j}\le 1-l_n x),$$
where $$j_2=\max\{j: c_{n, v}(1-l_n x)\ge 2j+v-5/2\}.$$ 
Lemma \ref{lelower1} again brings that 
$$\aligned \sum_{j=j_2+1}^n \log\mathbb{P}(X_{j}\le 1-l_n x)&=\sum_{j=j_2+1}^n\left(-2j\log \frac{j}{n(1-l_n x)}+2j-2n (1-l_n x)\right)+\tilde{O}(n\log n) \\
&=3(n^2-j_2^2)/2-2 n(n-j_2)(1-l_n x)+n^2\log (1-l_n x)\\
&\quad +j_2^2\log \frac{j_2}{n(1-l_n x)}+\tilde{O}(n\log n).
\endaligned $$
Considering the choice of $j_2,$ we have 
$$j_2=n(1-l_n x)+O(1).$$ 
Consequently, we have 
$$\aligned \sum_{j=1}^n \log\mathbb{P}(X_{j}\le 1-l_n x)&=n^2 l_n x+\frac12n^2l_n^2 x^2+n^2 \log(1-l_n x)+\tilde{O}(n\log n).
\endaligned $$
The condition $l_n<\!<1,$ and the Taylor formula again indicate 
$$n^2\log (1-l_n x)=- n^2 l_n x -\frac{n^2l_n^2 x^2}{2}-\frac{n^2l_n^3x^3}{3}+O(n^2l_n^3). $$
Finally, we get 
$$\aligned \sum_{j=1}^n \log\mathbb{P}(X_{j}\le 1-l_n x)&=-\frac{n^2l_n^3x^3}{3}+\tilde{O}(n\log n)+o(n^2 l_n^3).
\endaligned $$ 
Since $l_n^3>\!>\log n/n,$ it holds $n^2l_n^3>\!> n\log n.$ Eventually, we get 
$$\lim_{n\to\infty}\frac{1}{n^2l_n^3}\log\mathbb{P}(X_{(n)}\le 1-l_n x)=-\frac{x^3}{3}.$$
\subsection{Proof of Theorem \ref{extreminmdp} when $v$ bounded} 
 We see, 
 \begin{equation}\label{sum3}\lim_{n\to\infty}\frac1{n^2 l_n^2}\log \mathbb{P}(X_{(1)}\ge l_n \,x)=\lim_{n\to\infty}\frac1{n^2l_n^2}\sum_{j=1}^n\log \mathbb{P}(X_{j}\ge l_n \,x).\end{equation}
 As for Theorem \ref{extremindp}, we need to consider two cases: either $2j+v-1/2\le c_{n, v}l_n x$ or $2j+v-1/2>c_{n, v}l_n x.$
Set $$j_3=\max\{j: 2j+v-1/2\le c_{n, v} l_n \, x\}.$$ 
Since for any $j\ge j_3+1,$ 
$$\log \mathbb{P}(X_j\ge l_n x)=\tilde{O}(\log n),$$ 
we see 
\begin{equation}\label{sum4}\sum_{j=j_3+1}^n \log \mathbb{P}(X_j\ge l_n \, x)=\tilde{O}(n \log n)=o(n^2 l_n^2).\end{equation}
Here, we use the condition on $l_n,$ which ensures $$\frac{n\log n}{n^2 l_n^2}=\frac{\log n}{n l_n^2}=o(1)$$ for $n$ large enough.
On the other hand, Lemma \ref{asymlem1}, with $a_{n, v}=l_n\, x,$ implies that 
\begin{equation}\label{sum5} \aligned \sum_{j=1}^{j_3}\log \mathbb{P}(X_{j}\ge l_n \,x)&=\sum_{j=1}^{j_3}\left(-2j\log \frac{j}{n l_n x}+2j-2n l_n\, x\right)+\tilde{O}(j_3 \log n)\\
&=\frac{3 j_3^2}{2}-j_3^2\log \frac{j_3}{n l_n x}-2j_3 n l_n x+\tilde{O}(j_3\log n)\\
&=-\frac{n^2 l_n^2 x^2}{2}+o(n^2 l_n^2). 
\endaligned\end{equation}
Here the last equality is true based on the same reason as \eqref{sum2} and the fact $j_3=n l_n x+O(1).$ 
Plugging \eqref{sum4} and \eqref{sum5} into \eqref{sum3}, we have 
$$\lim_{n\to\infty}\frac{1}{n^2 l_n^2}\log \mathbb{P}(X_{(1)}\ge l_n x)=-\frac{1}{2} x^2.$$  

\section{The case $v\to\infty$} 
When $v\to\infty,$ Lemma \ref{kp} presents the asymptotical behavior of $K_v(vy)$ uniformly on $y\in (0, +\infty),$
which guides us to consider the following form 
$$\aligned \mathbb{P}(X_{j}\ge a)&=\mathbb{P}\left(\frac{2n Y_j}{v}\ge \frac{2\sqrt{n(n+v)}}{v} a\right)=\tilde{Z}_j^{-1}\int_{c_{n, v} a/v}^{\infty} y^{2j+v-1} K_v(vy) dy. 
\endaligned $$
Here, $$\tilde{Z}_j=\int_0^{\infty} y^{2j+v-1} K_{v}(vy) dy=v^{-(2j+v)}\int_0^{\infty} y^{2j+v-1} K_{v}(y) dy =v^{-(2j+v)} Z_j.$$
Recall that $c_{n, v}=2\sqrt{n(n+v)}$ and $$\tau_j(y)=\sqrt{1+y^2}-\log(1+\sqrt{1+y^2})+\frac{1}{4v}\log(1+y^2)-\frac{2j-1}{v}\log y$$ 
for $y>0$ and $x_j$ is the unique minimizer of $\tau_j.$  
The third item of Lemma \ref{kp} says $$ y^{2j+v-1} K_v(vy)=\sqrt{\frac{\pi}{2v}} e^{-v\tau_j(y)}(1+O(v^{-1}))$$ uniformly on $y\in (0, +\infty).$ Thereby, the probability $\mathbb{P}(X_{j}\ge a)$ could be written as  
\begin{equation}\label{keyex}\aligned \mathbb{P}(X_{j}\ge a)&=(1+o(1))\sqrt{\frac{\pi}{2}} v^{2j+v-1/2} Z_j^{-1} \int_{c_{n, v} a/v}^{+\infty}e^{-v\tau_j(y)} dy\\
&=\frac{2^{v+1/2}v^{2j+v-1/2}\Gamma(j+v+1/2)}{\,\Gamma(2j+2v)\,\Gamma(j)}\int_{c_{n, v} a/v}^{+\infty}e^{-v\tau_j(y)} dy(1+o(1))
\endaligned \end{equation} 
and similarly, 
\begin{equation}\label{keyex1}\aligned \mathbb{P}(X_{j}\le a)=\frac{2^{v+1/2}v^{2j+v-1/2}\Gamma(j+v+1/2)}{\,\Gamma(2j+2v)\,\Gamma(j)}\int_0^{c_{n, v} a/v}e^{-v\tau_j(y)} dy(1+o(1)).
\endaligned \end{equation} 

As in the preceding section, we first give an important lemma on 
$\log\mathbb{P}(X_j\ge a)$ when $a>0$. Since $v$ tends to the positive infinity in this section, we will use $a_{n, v}$ and $l_{n, v},$ instead of $a_n, l_n,$ respectively, to emphasize the dependence and to unify the notations.     

\begin{lem}\label{asymlem2} Consider the random variables $(X_j).$ Given a sequence $(a_{n, v})$ positive and bounded, which makes 
$\sqrt{n}\, a_{n, v}\to\infty$ as $n\to\infty.$ 
\begin{enumerate}
\item For the $j$ satisfying $c_{n, v} a_{n, v}/v>x_j,$ it holds that
$$\aligned \log \mathbb{P}(X_j\ge a_{n, v})&=(2j+v-1)\log v+(2j+v)(1-\log 2)-(j+v-1/2)\log(j+v)\\
 &\quad -(j-1/2)\log j-v \, u(c_{n, v} a_{n, v}/v)+(2j-1)\log(c_{n, v}a_{n, v}/v)+\tilde{O}(\log n). \endaligned $$
\item By contrast, for $j$ such that $c_{n, v} a_{n, v}/v\le x_j,$ we have that
$$\log \mathbb{P}(X_j\ge a_{n, v})=\tilde{O}(\log n).$$
\end{enumerate}
\end{lem}

Similarly, we give the asymptotic on $\log \mathbb{P}(X_j\le a_{n, v}).$ 

\begin{lem}\label{lelowerinf} Consider the same situations as Lemma \ref{asymlem2}. 
\begin{enumerate}
\item For $j$ satisfying $c_{n, v} a_{n, v}/v>x_j,$ it holds that
$$\log \mathbb{P}(X_j\le a_{n, v})=\tilde{O}(\log n).$$
 
\item By contrast, for $j$ such that $c_{n, v} a_{n, v}/v\le x_j,$ we have that
$$\aligned \log \mathbb{P}(X_j\le a_{n, v})&=(2j+v-1)\log v+(2j+v)(1-\log 2)-(j+v-1/2)\log(j+v)\\
 &\quad -(j-1/2)\log j-v \, u(c_{n, v} a_{n, v}/v)+(2j-1)\log(c_{n, v}a_{n, v}/v)+\tilde{O}(\log n). \endaligned $$
\end{enumerate}
\end{lem}

The proofs of Lemmas \ref{asymlem2} and \ref{lelowerinf} are relative lengthy, and they are postponed into the appendix. With these two lemmas at hand, next we are going to verify the theorems one by one. 
\subsection{Proof of Theorem \ref{extremaxdp} when $v\to\infty$.} 
We still claim that
$$\lim_{n\to\infty}\frac{1}{n}\log\mathbb{P}(X_{(n)}\ge x)=\lim_{n\to\infty}\frac{1}{n}\log\mathbb{P}(X_{n}\ge x).$$ 
Since $x>1,$ the second item in Lemma \ref{keyl} entails that
$$\frac{c_{n, v} x}{v x_n}\ge \frac{2\sqrt{n(n+v)} \; x}{2\sqrt{(n-1/2)(n+v-1/2)}}>1.$$ 
Therefore, Lemma \ref{asymlem2}, with $a_{n, v}\equiv x$ and $j=n,$ claims that   
$$\aligned \log \mathbb{P}(X_n\ge x)=-v u\left(\frac{c_{n, v} x}{v}\right)+2n+v-v\log\frac{2(n+v)}{v}+n\log(x^2)+o(n).
\endaligned $$
It follows from the condition $v/n\to\alpha$ that 
$$\lim_{n\to\infty}\frac1n\log\mathbb{P}(X_n\ge x)=-\alpha u\left(\frac{2\sqrt{1+\alpha} \,x}{\alpha}\right)+2+\alpha-\alpha\log\left( 2\left(1+ \frac{1}{\alpha} \right)\right)+\log(x^2).$$
Recall $u(y)=\sqrt{1+y^2}-\log(1+\sqrt{1+y^2})$ and observe the following relationship 
$$\sqrt{1+\frac{4(1+\alpha) x^2}{\alpha^2}}-1=\frac{2\kappa_{\alpha}(x)}{\alpha}.$$ Then,  \begin{equation}\label{cofu} u\left(\frac{2\sqrt{1+\alpha} \,x}{\alpha}\right)=1-\log 2+\frac{2 \,\kappa_{\alpha}(x)}{\alpha}-\log\left(1+\frac{\kappa_{\alpha}(x)}{\alpha}\right). \end{equation}
Consequently, we have that
$$\lim_{n\to\infty}\frac1n\log\mathbb{P}(X_{(n)}\ge x)=\lim_{n\to\infty}\frac1n\log\mathbb{P}(X_{n}\ge x)=-I_{\alpha}^{(r)}(x)$$
for $\alpha\in (0, \infty).$ For $\alpha=\infty,$ applying the Taylor formula to $v/n\log(1+n/v)$ and $u(c_{n, v} x/v),$ we will have $I_{\infty}^{(r)}(x)$ replacing $I_{\alpha}^{(r)}.$ The proof is similar when $\alpha=0,$ and is omitted here.  

Now we suppose that $0<x<1.$ Define $j_4$ as 
\begin{equation}\label{j4} j_4=\max\{j: c_{n, v} x>v x_{j}\}.\end{equation}
Lemma \ref{lelowerinf} again says that 
$$\log\mathbb{P}(X_j\le x)=\tilde{O}(\log n)
$$
for any $j\le  j_4$ and 
$$\aligned \log \mathbb{P}(X_j\le x)&=(2j+v-1)\log v+(2j+v)(1-\log 2)-(j+v-1/2)\log(j+v)\\
 &\quad -(j-1/2)\log j-v \, u(c_{n, v} x/v)+(2j-1)\log(c_{n, v} x/v)+\tilde{O}(\log n)\\
 &=:d_{n, v}(j, x)+o(n)\endaligned $$
for any $j>j_4.$ Starting from these estimates, we might have the following simple equality  
$$\log\mathbb{P}(X_{(n)}\le x)=\sum_{j=1}^n \log\mathbb{P}\left(X_{j}\le x\right)= \sum_{j=j_4+1}^n d_{n, v}(j, x)+o(n^2).$$ 
With the help of Lemma \ref{ma}, summing these terms up, we get  simultaneously that 
$$\aligned \sum_{j=1}^n d_{n, v}(j, x)&=(n^2+nv)\log v+(n^2+n v)(1-\log 2)-v n \, u\left( \frac{c_{n, v} x}{v}\right)+n^2\log\left( \frac{c_{n, v} x}{v}\right)\\
&\quad +\frac{n^2+nv}{2}-\frac{(n+v)^2}{2}\log(n+v)+\frac{v^2}{2}\log v-\frac{n^2}{2}\log n+o(n^2); \\
 \sum_{j=1}^{j_4} d_{n, v}(j, x)&=(j_4^2+j_4 v)\log v+(j_4^2+j_4 v)(1-\log 2)-v j_4\, u\left( \frac{c_{n, v} x}{v}\right)+j_4^2\log\left( \frac{c_{n, v} x}{v}\right)\\
&\quad +\frac{j_4^2+j_4 v}{2}-\frac{(j_4+v)^2}{2}\log(j_4+v)+\frac{v^2}{2}\log v-\frac{j_4^2}{2}\log j_4+o(n^2).
\endaligned $$
Some algebras on the difference between these two sums enable us to get that 
\begin{equation}\label{leftdpe}\aligned \sum_{j=j_4+1}^n d_{n, v}(j, x)=&-v(n-j_4)\left(u \left( \frac{c_{n, v} x}{v}\right)-3/2+\log 2+\log\left(1+\frac{n}{v} \right)\right)+\frac{j_4^2}{2}\log \frac{j_4}{n}\\
&+(n^2-j_4^2)\left(\log x+ \frac 3 2 \right)+\frac{(j_4+v)^2}{2}\log\frac{j_4+v}{n+v}+o(n^2).
\endaligned \end{equation}
Suppose $\lim\limits_{n\to\infty} j_4/n=a.$ We go back to Lemma \ref{keyl} and the definition of $j_4$ to figure out that $$c_{n, v} x=2\sqrt{j_4(j_4+v)}(1+O(j_4^{-1})),$$ which indicates $a(a+\alpha)=(1+\alpha) x^2$ and then $$a=\frac{2(1+\alpha)x^2}{\alpha+\sqrt{\alpha^2+4(1+\alpha) x^2}}=\kappa_{\alpha}(x).$$ 
The condition $\lim\limits_{n\to\infty}v/n=\alpha$ and the relationship \eqref{cofu} finally bring us the desired limit  
$$\aligned &\quad \lim_{n\to\infty}\frac{1}{n^2}\sum_{j=1}^{n} \log\mathbb{P}(X_j\le x)\\
&=-\alpha(1-\kappa_{\alpha}(x))\left(\log\frac{1+\alpha}{\kappa_{\alpha}(x)+\alpha}+\frac{2\kappa_{\alpha}(x)}{\alpha}\right)+\frac{\alpha(1-\kappa_{\alpha}(x))}{2}\\
&\quad +\frac{\kappa_{\alpha}(x)^2}{2}\log\kappa_{\alpha}(x)+(1-\kappa_{\alpha}^2(x))\left(\log x+\frac{3}{2}\right)+\frac{(\alpha+\kappa_{\alpha}(x))^2}{2}\log\frac{\kappa_{\alpha}(x)+\alpha}{1+\alpha}\\
&=-\left(\alpha+\frac{\alpha^2}{2}\right)\log\frac{1+\alpha}{\kappa_{\alpha}(x)+\alpha}+\log x+\frac{(\alpha+3-\kappa_{\alpha}(x))(1-\kappa_{\alpha}(x))}{2}. \endaligned $$
Here, for the last equality we use the relationship $\kappa_{\alpha}(x)(\alpha+\kappa_{\alpha}(x))=(1+\alpha)x^2.$
Suppose that $\alpha=+\infty,$ under which $\kappa_{\infty}(x)=x^2.$ Recalling the Taylor formula $\log(1+x)=x-x^2/2+o(x^2),$ one gets that 
$$-\log\frac{1+\alpha}{\kappa_{\alpha}(x)+\alpha}=\frac{\kappa_{\alpha}(x)-1}{\alpha}+\frac{1-\kappa_{\alpha}^2(x)}{2\alpha^2}+o(\alpha^{-2}).$$
Hence, 
$$\lim_{n\to\infty}\frac{1}{n^2}\sum_{j=1}^{n} \log\mathbb{P}(X_j\le x)=\log x+\frac{x^4-4x^2+3}{2}=-I_{\infty}^{(l)}(x).$$ 
Similarly, when $\alpha=0,$ it holds that  
$$\lim_{n\to\infty}\frac{1}{n^2}\sum_{j=1}^{n} \log\mathbb{P}(X_j\le x)=\log x+\frac{x^2-4x+3}{2}=-I_{0}^{(l)}(x).$$
The left deviation probability is obtained. 
 \subsection{Proof of Theorem \ref{extremindp} when $v\to\infty$} 
  
We first work on the simpler case $x\ge 1.$ Observe that when $x\ge 1,$ the following inequality always holds 
$$\frac{c_{n, v} x}{v x_j}\ge \frac{2\sqrt{n(n+v)} \; x}{2\sqrt{(j-1/2)(j+v-1/2)}}>1$$
for any $1\le j\le n.$   
Hence, apply Lemma \ref{asymlem2} again, with $a_{n, v}\equiv x,$ to rewrite   
 \begin{equation}\label{logpform}
 	\aligned \log\mathbb{P}(X_j\ge x)&=(2j+v-1)\log v+(2j+v)(1-\log 2)-(j+v-1/2)\log(j+v)\\
 &\quad -j\log j-v \, u\left(\frac{c_{n, v} x}{v}\right)+(2j-1)\log\left(\frac{c_{n, v} x}{v}\right)+o(n).
\endaligned  \end{equation}
Lemma  \ref{ma} and simple calculus imply that   
\begin{equation}\label{sumofp}
	\aligned \sum_{j=1}^n\log\mathbb{P}(X_j\ge x)&=-\frac{v^2+2n v}{2}\log\left(1+\frac{n}{v}\right)+(n^2+nv)\left(\frac 3 2-\log 2\right)\\
&\quad -v\,n \, u\left(\frac{c_{n, v} x}{v}\right)+n^2\log(2x)+o(n^2).
\endaligned \end{equation}
The condition $\lim\limits_{n\to\infty}v/n=\alpha$ and \eqref{cofu} help us eventually to arrive at
$$\aligned \lim_{n\to\infty}\frac1{n^2}\sum_{j=1}^n\log\mathbb{P}(X_j\ge x)&=-\frac{\alpha(\alpha+2)}{2}\log\left(1+\frac1{\alpha}\right)+\frac{3+\alpha}2+\log x-2\,\kappa_{\alpha}(x)\\
&\quad +\alpha\log\left(1+\frac{\kappa_{\alpha}(x)}{\alpha}\right)=- \widehat{J}_{\alpha}(x). 
\endaligned $$  
 The rate function $ \widehat{J}_0$ and $ \widehat{J}_{\infty}$ are the limits of $ \widehat{J}_{\alpha}$ related to $\alpha=0$ or $\alpha=\infty,$ respectively. 

Now we suppose that $0<x<1.$   
Let $j_4$ be given in \eqref{j4}. For any $1\le j\le j_4,$ since $c_{n, v} x/v>x_{j},$ the same argument as for \eqref{logpform} leads that   
 \begin{equation}\label{logpform1}
 	\aligned \log\mathbb{P}(X_j\ge x)=&(2j+v-1)\log v+(2j+v)(1-\log 2)-(j+v-1/2)\log(j+v)\\
 &-j\log j-v \, u\left(\frac{c_{n, v} x}{v}\right) +(2j-1)\log\left(\frac{c_{n, v} x}{v}\right)+o(n).
\endaligned  \end{equation}
Meanwhile, Lemma \ref{asymlem2} also implies the following expression 
$$\sum_{j_4+1}^n\log\mathbb{P}(X_j\ge x)=o(n^2).$$ 
Therefore, similarly as \eqref{sumofp}, we know  that
\begin{equation}\aligned \label{sumuptoj3}
\sum_{j=1}^{n}\log\mathbb{P}(X_j\ge x)&=\sum_{j=1}^{j_4}
\log\mathbb{P}(X_j\ge x)+\sum_{j=j_4+1}^n\log\mathbb{P}(X_j\ge x)\\
&=-\frac{v^2+2j_4 v}{2}\log \left(1+\frac{j_4}{v}\right)+(j_4^2+j_4 v)\left(\frac{3}{2}-\log 2\right)\\
&\quad -v\,j_4 \, u\left(\frac{c_{n, v} x}{v}\right)-\frac{j_4^2}{2}\log\frac{j_4(j_4+v)}{c_{n, v}^2 x^2}+o(n^2).
\endaligned \end{equation} 
Consequently,    
$$\aligned
&\ \lim_{n\to\infty}\frac{1}{n^2}\log\mathbb{P}(X_j\ge x)\\
=&-\frac{\alpha^2+2\kappa_{\alpha}(x)\alpha}{2}\log\left(1+\frac{\kappa_{\alpha}(x)}{\alpha}\right)+(\kappa_{\alpha}^2(x)+\kappa_{\alpha}(x)\alpha)\left(\frac 3 2-\log 2\right)\\
&-\alpha\kappa_{\alpha}(x)u\left(\frac{2\sqrt{1+\alpha} \;x}{\alpha}\right)+\frac{\kappa^2_{\alpha}(x)}2\log\frac{4(1+\alpha)x^2}{\kappa_{\alpha}(x)(\alpha+\kappa_{\alpha}(x))}.  \endaligned$$  
Using again the expression \eqref{cofu} and the fact $\kappa_{\alpha}(x)(\alpha+\kappa_{\alpha}(x))=(1+\alpha)x^2,$ and combining like terms, we have the following limit as
$$ \lim_{n\to\infty}\frac{1}{n^2}\log\mathbb{P}(X_j\ge x)=-\frac{\alpha^2}{2}\log\left(1+\frac{\kappa_{\alpha}(x)}{\alpha}\right)+\frac{\alpha\kappa_{\alpha}(x)-\kappa_{\alpha}^2(x)}{2} $$
for $\alpha\in (0, +\infty)$ and the limits related to $\alpha=0$ and $\alpha=+\infty$ follow naturally by taking corresponding limits. 

\subsection{Proof of Theorem \ref{extremaxmdp} when $v\to\infty$} 
We prove first the right moderate deviation probability of $X_{(n)}.$ By the conditions on $l_{n, v}$ in the first item, we see both the fact $n l_{n, v}^2>\!>\log n$ and the inequality
$c_{n, v}(1+l_{n, v} x)>x_j$ for any $1\le j\le n$ when $x>0.$ 
Thus $$\lim_{n\to\infty}\frac{1}{n l_{n, v}^2}\log\mathbb{P}(X_{(n)}\ge 1+l_{n, v} x)=\lim_{n\to\infty}\frac{1}{n l_{n, v}^2}\log\mathbb{P}(X_{n}\ge 1+l_{n, v} x).$$
 It follows from Lemma \ref{asymlem2}, with $a_{n, v}=1+l_{n, v} x,$ that  
 \begin{align} &\quad \log\mathbb{P}(X_{n}\ge 1+l_{n, v} x)\notag\\
 &=(2n+v-1)\log v+(2n+v)(1-\log 2)-(n+v-1/2)\log(n+v)\notag\\
 &\quad -n\log n-v \, u(c_{n, v} (1+l_{n, v} x)/v)+(2n-1)\log(c_{n, v}(1+l_{n, v} x)/v)+\tilde{O}(\log(n)). \label{rightlog}
 \end{align} The fact $c_{n, v}=2\sqrt{n(n+v)}$ allows us to write 
 $$\aligned &\quad (2n-1)\log(c_{n, v}(1+l_{n, v} x)/v)\\
 &=(n-\frac12)\log n+
 (n-\frac12)\log(n+v)+(2n-1)\log(2(1+l_{n, v} x))-(2n-1)\log v. 
 \endaligned $$  
 Putting this expression into \eqref{rightlog}, combining alike terms while paying attention to terms relating to $v$ and ignoring the terms with order less than  $\log n,$ one gets that 
 $$\aligned &\quad \log\mathbb{P}(X_{n}\ge 1+l_{n, v} x)+v \, u(c_{n, v} (1+l_{n, v} x)/v)\\
 &=-v u(c_{n, v} (1+l_{n, v} x)/v)+v\log \frac{v}{2(n+v)}+v+2n+2n\log(1+l_{n, v } x)+O(\log n). \endaligned$$ 
 Since $u(x)=\sqrt{1+x^2}-\log(1+\sqrt{1+x^2}),$ setting $d_{n, v}^{(+)}:=\sqrt{v^2+c_{n, v}^2(1+l_{n, v} x)^2},$ we write 
\begin{equation}\label{examp}\aligned v u\left(\frac{c_{n, v} (1+l_{n, v} x)}{v}\right)-v-v\log\frac{v}{2(n+v)}&=d_{n, v}^{(+)}-v-v\log\frac{v+d_{n, v}^{(+)}}{2(n+v)}.
 \endaligned \end{equation}
 Hence, it follows from the Taylor formula \begin{equation}\label{taylor}\log(1+x)=x-\frac{x^2}{2}+O(x^3)\end{equation} for $|x|$ small enough and $n\l_{n, v}^2\gg \log n$ that  
 \begin{equation}\label{s0} \aligned
 	&\quad\log\mathbb{P}(X_{n}\ge 1+l_{n, v} x)\\
 	&=2n+v-d_{n,v}^{(+)}+v\log \left(1-\frac{2n+v-d_{n, v}^{(+)}}{2(n+v)}\right)+2n l_{n, v} x-n l_{n, v}^2x^2+O(n l_{n, v}^3). \endaligned\end{equation} 
 	By definition, 
 	$$d_{n, v}^{(+)}=(v+2n)\sqrt{1+\frac{c_{n, v}^2(2 l_{n, v} x+l_{n, v}^2 x^2)}{(v+2n)^2}}.$$ 
 	Since $c_{n, v}^2=4n(n+v)\le (v+2n)^2$ and $0<l_{n, v}\ll 1,$ it holds that  
 	$$c_{n, v}^2(2 l_{n, v} x+l_{n, v}^2 x^2)/(v+2n)^2=O(c_{n, v}^2 l_{n, v}/(v+2n)^2)=\widetilde{O}(l_{n, v})=o(1).$$   
 Using the Taylor formula $$\sqrt{1+x}=1+\frac x 2- \frac{x^2}{8}+O(x^3)$$ for $|x|$ small enough and ignoring the term $o(n l_{n, v}^2)$, we see that
 $$\aligned  
 d_{n, v}^{(+)}&=(v+2n)\sqrt{1+\frac{c_{n, v}^2(2 l_{n, v} x+l_{n, v}^2 x^2)}{(v+2n)^2}}\\
 &=(v+2n)\left(1+\frac{c_{n, v}^2 (2l_{n, v} x+l_{n, v}^2 x^2)}{2(v+2n)^2}-\frac{c_{n, v}^4(2l_{n, v}x+l_{n, v}^2 x^2)^2}{8(v+2n)^4}+O\left(\frac{c_{n, v}^6 l_{n, v}^3}{(v+2n)^6}\right)\right)\\
 &=(v+2n)\left(1+\frac{c_{n, v}^2 l_{n, v} x}{(v+2n)^2}+\frac{c_{n, v}^2 l_{n, v}^2 x^2}{2(v+2n)^2}-\frac{c_{n, v}^4l_{n, v}^2 x^2}{2(v+2n)^4}+O\left(\frac{c_{n, v}^6 l_{n, v}^3}{(v+2n)^6}\right)\right)\\
 &=(v+2n)\left(1+\frac{c_{n, v}^2 l_{n, v} x}{(v+2n)^2}+\frac{c_{n, v}^2 v^2 l_{n, v}^2 x^2}{2(v+2n)^4}\right)+\widetilde{O}(n l_{n, v}^3),
 \endaligned $$  
 where the last equality holds because 
 $$\frac{c_{n, v}^6 l_{n, v}^3}{(v+2n)^5}=O\left(\frac{n^3 l_{n, v}^3}{(v+2n)^2}\right)=\widetilde{O}(n l_{n, v}^3).$$ 
 Thus, 
 \begin{equation}\label{s2} g_{n, v}^{(+)}:=2n+v-d_{n, v}^{(+)}=-\frac{c_{n, v}^2 l_{n, v} x}{v+2n}-\frac{c_{n, v}^2 v^2 l_{n, v}^2 x^2}{2(v+2n)^3}+\widetilde{O}(n l_{n, v}^3),\end{equation}
 which implies that
 $$ \frac{g_{n, v}^{(+)}}{n+v}=O\left(\frac{nl_{n, v}}{v+2n}\right)=o(1).$$
Starting from this expression, with the help of \eqref{taylor} again, and just picking up the terms relating to $l_{n, v}$ and $l_{n, v}^2$ and examining their coefficients, we obtain that \begin{equation}\label{s1}
\aligned v\log\left(1-\frac{2n+v-d_{n, v}^{(+)}}{2(n+v)}\right)&=v
\log\left(1-\frac{g_{n, v}^{(+)}}{2(n+v)}\right)\\
&=-v\left(\frac{g_{n, v}^{(+)}}{2(n+v)}+\frac{(g_{n, v}^{(+)})^2}{8(n+v)^2}\right)+O\left(\frac{v n^3 l_{n, v}^3}{(v+2n)^3}\right)\\
&=\frac{2n v l_{n, v} x}{v+2n}+l_{n, v}^2 x^2\left(\frac{n v^3}{(v+2n)^3}-\frac{2n^2 v}{(v+2n)^2}\right)+\widetilde{O}(n l_{n, v}^3). 
\endaligned \end{equation} 
Putting \eqref{s2} and \eqref{s1} into \eqref{s0}, we get that
$$\aligned\log\mathbb{P}(X_n\ge 1+l_{n, v} x)&=\frac{2nv l_{n, v} x}{v+2n}+l_{n, v}^2 x^2\left(\frac{n v^3}{(v+2n)^3}-\frac{2n^2 v}{(v+2n)^2}\right)-\frac{c_{n, v}^2 l_{n, v} x}{v+2n}\\
&\quad-\frac{c_{n, v}^2 v^2 l_{n, v}^2 x^2}{2(v+2n)^3}+2n l_{n, v} x-n l_{n, v}^2 x^2+O(n l_{n, v}^3)\\
&=-\frac{2n(n+v)}{2n+v}l_{n, v}^2 x^2+o(n l_{n, v}^2).\endaligned$$
 Hence, the condition $\lim\limits_{n\to\infty}\frac{v}{n}= \alpha$ helps us to get the following limit
 $$\lim_{n\to\infty}\frac{1}{n l_{n, v}^2}\log\mathbb{P}(X_{(n)}\ge 1+l_{n, v} x)=-\frac{2(1+\alpha)}{2+\alpha} x^2.$$ 
 
 The proof of the left moderate deviation in Theorem \ref{extremaxmdp} is similar to that of 
 the second item of Theorem \ref{extremaxdp}. We define 
 $$j_5=\max\{j: c_{n, v} (1-l_{n, v} x) /v>x_j\}.$$  
Lemma \ref{lelowerinf} again, with $a_{n, v}=1-l_{n, v} x,$ ensures that  
$$\log\mathbb{P}(X_j\le 1-l_{n, v} x)=\tilde{O}(\log n)
$$
for any $j\le  j_5,$ while by contrast  
$$\aligned &\quad \log \mathbb{P}(X_j\le 1-l_{n, v} x)
\\
&=(2j+v-1)\log v+(2j+v)(1-\log 2)-(j+v-1/2)\log(j+v)-(j-1/2)\log j\\
 &\quad -v \, u(c_{n, v}(1-l_{n, v} x)/v)+(2j-1)\log(c_{n, v}(1-l_{n, v} x)/v)+\tilde{O}(\log n)\\
 \endaligned $$ 
 for any $j>j_5.$
 Hence, we have that  
 \begin{equation}\label{leftmoderates}\log\mathbb{P}(X_{(n)}\le 1-l_{n, v} x)=\sum_{j=j_5+1}^n \log\mathbb{P}(X_j\le 1-l_{n, v} x)+\tilde{O}(n\log n).\end{equation}
Replacing $x$ and $j_4$ in the expression \eqref{leftdpe} by $1-l_{n, v} x$ and $j_5,$ respectively, we have   
\begin{equation}\label{leftmoderate}\aligned 
& \sum_{j=j_5+1}^n \log\mathbb{P}(X_j\le 1-l_{n, v} x)\\
=&-v(n-j_5)\left(u\left(\frac{c_{n, v}\left(1-l_{n, v} x\right)}{v}\right)- \frac 3 2+\log 2+\log\left(1+\frac nv\right)\right)+\frac{j_5^2}{2}\log \frac{j_5}{n}\\
&+(n^2-j_5^2)\left(\log (1-l_{n, v} x)+\frac32\right)+\frac{(j_5+v)^2}{2}\log\frac{j_5+v}{n+v}+\tilde{O}(n\log n).
\endaligned \end{equation} 
Recall the condition  
 $$\log n/n<\!< l_{n, v}^3<\!<1,$$ which implies $n\log n=o(n^2 l_{n, v}^3).$ 
It follows from \eqref{leftmoderates} and \eqref{leftmoderate} that 
\begin{equation}\label{sumsum}\aligned &\quad\log\mathbb{P}(X_{(n)}\le 1-l_{n, v} x)\\
&=-v(n-j_5)\left(u\left(\frac{c_{n, v}\left(1-l_{n, v} x\right)}{v}\right)- \frac 3 2+\log 2+\log\left(1+\frac nv\right)\right)+\frac{j_5^2}{2}\log \frac{j_5}{n}\\
&\quad+(n^2-j_5^2)\left(\log (1-l_{n, v} x)+\frac{3}{2} \right)+\frac{(j_5+v)^2}{2}\log\frac{j_5+v}{n+v}+o(n^2l_{n, v}^3).\endaligned\end{equation}
We need treat the term $u(c_{n, v} (1-l_{n, v} x)/v).$ 
As above, similarly set 
$$d_{n, v}^{(-)}:=\sqrt{v^2+c_{n, v}^2(1-l_{n, v} x)^2} \quad \text{and} \quad g_{n, v}^{(-)}:=2n+v-d_{n, v}^{(-)}.$$ 
Using instead the following Taylor formula 
$$\sqrt{1+x}=1+\frac{x}{2}-\frac{x^2}{8}+\frac{x^3}{16}+o(x^3),$$ applying the relationship $c_{n, v}^2+v^2=(v+2n)^2$ and ignoring terms relating to $l_{n, v}^k$ for $k\ge 4,$ one obtains the following asymptotic  
$$\aligned \frac{d_{n, v}^{(-)}}{v+2n}&=\sqrt{1+\frac{c_{n, v}^2(-2 l_{n, v} x+l_{n, v}^2 x^2)}{(v+2n)^2}}&\\
&=\left(1-\frac{c_{n, v}^2 l_{n, v} x}{(v+2n)^2}+\frac{c_{n, v}^2 v^2 l_{n, v}^2 x^2}{2(v+2n)^4}+\frac{c_{n, v}^4 v^2 l_{n, v}^3 x^3}{2(v+2n)^6}\right)+o\left(\frac{nl_{n, v}^3}{v+2n}\right)\endaligned $$ 
and then 
$$g_{n, v}^{(-)}=\frac{c_{n, v}^2 l_{n, v} x}{v+2n}-\frac{c_{n, v}^2 v^2 l_{n, v}^2 x^2}{2(v+2n)^3}-\frac{c_{n, v}^4 v^2 l_{n, v}^3 x^3}{2(v+2n)^5}+o(n l_{n, v}^3).$$ 
These two asymptotics imply the following expression
\begin{equation}\label{sfor1}\aligned & \quad v \, u(c_{n, v}(1-l_{n, v} x)/v)-(v+2n)-v\log\frac{v}{2(n+v)}\\
&=-g_{n, v}^{(-)}-v\log\left(1-\frac{g_{n, v}^{(-)}}{2(n+v)}\right)\\
&=-g_{n, v}^{(-)}+v\left(\frac{g_{n, v}^{(-)}}{2(n+v)}+\frac{(g_{n, v}^{(-)})^2}{8(n+v)^2}+\frac{(g_{n, v}^{(-)})^3}{24(n+v)^3}\right)+o(nl_{n, v}^3)\\
&=-2n l_{n, v} x+\frac{n v l_{n, v}^2 x^2}{2n+v}+\frac{2n^2 v(3v+2n)l_{n, v}^3 x^3}{3(v+2n)^3}+o(n l_{n, v}^3).\endaligned \end{equation}
Putting \eqref{sfor1} back into \eqref{sumsum}, and combining alike terms, we obtain that
\begin{equation}\label{loglnv}\aligned &\quad \log\mathbb{P}(X_{(n)}\le 1-l_{n, v} x)\\
&=-2n(n-j_5)+\frac{(n-j_5)v}2+(n-j_5)\left(2n l_{n, v} x-\frac{nvl_{n, v}^2 x^2}{2n+v}-\frac{2n^2 v(3v+2n)l_{n, v}^3 x^3}{3(v+2n)^3}\right)\\
&\quad+(n^2-j_5^2)\log(1-l_{n, v} x)+\frac{3}{2}(n^2-j_5^2)+\frac{j_5^2}{2}\log\frac{j_5}{n}+\frac{(j_5+v)^2}{2}\log\frac{j_5+v}{n+v}+o(n^2 l_{n, v}^3).
\endaligned\end{equation}
The information on $j_5$ is obliged to be done for the desired limit corresponding to $\log\mathbb{P}(X_{(n)}\le 1-l_{n, v} x)$. By definition, with the help of the properties of $x_j$ and the expression of $g_{n, v}^{(-)}$ we know  that 
$$\aligned \frac{j_5}{n}&=\frac{-v+\sqrt{v^2+c_{n, v}^2(1-l_{n, v} x)^2}}{2n}+O(n^{-1})\\
&=-\frac{g_{n, v}^{(-)}}{2n}+1+O(n^{-1})\\
&=1-\frac{c_{n, v}^2 l_{n, v} x}{2n(v+2n)}+\frac{c_{n, v}^2 v^2 l_{n, v}^2 x^2}{4n(v+2n)^3}+\frac{c_{n, v}^4 v^2 l_{n, v}^3 x^3}{4n(v+2n)^5}+o(l_{n, v}^3)\\
&=1-\frac{2(v+n)l_{n, v} x}{2n+v}+\frac{(n+v) v^2 l_{n, v}^2 x^2}{(2n+v)^3}+\frac{4n(n+v)^2v^2l_{n, v}^3 x^3}{(v+2n)^5}+o(l_{n, v}^3). \endaligned $$ 
Set
$$e_{n, v}=\frac{2(v+n)}{2n+v}, \quad f_{n, v}=\frac{(n+v) v^2 }{(2n+v)^3} \quad \text{and}\quad h_{n, v}=\frac{4n(n+v)^2v^2}{(v+2n)^5}$$ to simplify  
\begin{equation}\label{j5f}\frac{j_5}{n}=1-e_{n, v}l_{n, v} x+f_{n, v} l_{n, v}^2 x^2+h_{n, v}l_{n, v}^3 x^3+o(l_{n, v}^3),\end{equation} 
which indicates that $j_5=n+O(n l_{n, v}).$ Further, applying the Taylor formula on $\log(1+x)$ to have $$
\aligned (n^2-j_5^2) \log(1-l_{n, v } x)&=-(n^2-j_5^2)\left(l_{n, v}x+\frac{l_{n, v}^2 x^2}{2}\right)+o(n^2 l_{n, v}^3)\\
&=-2n(n-j_5)\left(l_{n, v}x+\frac{l_{n, v}^2 x^2}{2}\right)+(n-j_5)^2l_{n, v}x+o(n^2 l_{n, v}^3) \endaligned $$ and     
$$\aligned \quad j_5^2\log\frac{j_5}{n}&=-(n^2+(n-j_5)^2-2n(n-j_5))\left(\frac{n-j_5}{n}+\frac{(n-j_5)^2}{2n^2}+\frac{(n-j_5)^3}{3n^3}\right)+o(n^2 l_{n, v}^3) \\
&=-n(n-j_5)+\frac32 (n-j_5)^2-\frac{(n-j_5)^3}{3n}+o(n^2 l_{n, v}^3)
\endaligned $$
and similarly 
$$\aligned
(j_5+v)^2\log\frac{j_5+v}{n+v}&=-(j_5-n+n+v)^2\left(\frac{n-j_5}{n+v}+\frac{(n-j_5)^2}{2(n+v)^2}+\frac{(n-j_5)^3}{3(n+v)^3}\right)+o(n^2 l_{n, v}^3)\\
&=-(n+v)(n-j_5)+\frac32 (n-j_5)^2-\frac{(n-j_5)^3}{3(n+v)}+o(n^2 l_{n, v}^3).
\endaligned $$ 
Plugging all these asymptotics into \eqref{loglnv}, 
arranging the terms according to different orders of $(n-j_5),$ keeping in mind that $n-j_5=O(n l_{n, v})$ and ignoring the terms belonging to $o(n^2 l_{n, v}^3),$ we are able to rewrite the expression \eqref{loglnv} for $\log\mathbb{P}(X_{(n)}\le 1-l_{n, v} x)$ 
as 
\begin{equation}\label{fsum}\aligned \log\mathbb{P}(X_{(n)}\le 1-l_{n, v} x)&=-(n-j_5)\left(\frac{nv l_{n, v}^2 x^2}{v+2n}+n l_{n, v}^2 x^2\right)+(n-j_5)^2l_{n, v} x\\
&\quad -(n-j_5)^3\left(\frac{1}{6n}+\frac{1}{6(n+v)}\right)+o(n^2 l_{n, v}^3)\\
&=-n e_{n, v}(n-j_5)l_{n, v}^2 x^2+(n-j_5)^2l_{n, v} x-\frac{(n-j_5)^3}{3n e_{n, v}}+o(n^2 l_{n, v}^3).
\endaligned\end{equation}
Inserting the expression \eqref{j5f} of $j_5/n$ into \eqref{fsum} and only keeping the terms up to $O(n^2 l_{n, v}^3)$, we know in further that 
$$\aligned \quad \log\mathbb{P}(X_{(n)}\le 1-l_{n, v} x)&=-n^2e_{n, v}^2 l_{n, v}^3 x^3+n^2 e_{n, v}^2 l_{n, v}^3 x^3-\frac{n^2 e_{n, v}^2 l_{n, v}^3 x^3}{3}+o(n^2 l_{n, v}^3)\\
&=-\frac{n^2 e_{n, v}^2 l_{n, v}^3 x^3}{3}+o(l_{n, v}^3).
\endaligned$$
Jointing the expressions of $e_{n, v},$ we obviously have that 
$$\lim_{n\to\infty}\frac1{n^2 l_{n, v}^3}\log\mathbb{P}(X_{(n)}\le 1-l_{n, v} x)=-\frac{4(1+\alpha)^2}{3(2+\alpha)^2} x^3,$$ 
which is the desired left moderate deviation.   
 
 \subsection{Proof of Theorem \ref{extreminmdp} when $v\to\infty$}
 A similar argument, as for Theorem \ref{extremindp}, will work here. For $x>0,$ define 
 $$j_6:=\max\{j: \; c_{n, v} l_{n, v} x/v>x_j\}.$$  
  It holds, due to Lemma \ref{asymlem2} with $a_{n, v}=l_{n, v} x,$ that 
 $$ \aligned 
 	\log\mathbb{P}(X_j\ge l_{n, v} x)
 	= & (2j+v-1)\log v+(2j+v)(1-\log 2)-(j+v-1/2)\log(j+v)\\
	&-j\log j-v \, u(c_{n, v} l_{n, v} x/v)+(2j-1)\log(c_{n, v} l_{n, v} x/v)+\tilde{O}(\log n)\\
\endaligned $$ 
for $j\le j_6,$ and 
$$\log\mathbb{P}(X_j\ge l_{n, v} x)=\tilde{O}(\log n )$$ 
for the case when $j>j_6.$
Lemma \ref{ma} helps us to get that 
 \begin{equation}\label{minsum}\aligned 
 	& \sum_{j=1}^{j_6}\log\mathbb{P}(X_j\ge l_{n, v} x) \\
 	= & (j_6^2+j_6 v)\log v+(j_6^2+j_6 v)(1-\log 2)+\frac{j_6^2+j_6 v}{2}-\frac{(j_6+v)^2}{2}\log(j_6+v)\\
 &+\frac{v^2}{2}\log v-\frac{j_6^2}{2}\log j_6+j_6^2\log\frac{c_{n, v} l_{n, v} x}{v} -v\,j_6 \, u\left(c_{n, v} l_{n, v} x/v\right)+o(j_6\log n).\\
 \endaligned \end{equation}
 Detailed information on $j_6$ is necessary.   
 Combining the definition, and the fact that
 $$2\sqrt{(j-3/4)(j+v-3/4)} \le v x_j\le 2\sqrt{(j-1/4)(j+v-1/4)},$$
 we know that \begin{equation}\label{j6l} 4(j_6-1)(j_6+v-1)<c_{n, v}^2 l_{n, v}^2 x^2<4(j_6+1)(j_6+v+1).\end{equation}
 Thus, 
 \begin{equation}\label{j6p}j_6=(\sqrt{v^2+c_{n, v}^2 l_{n, v}^2 x^2}-v)/2+O(1)\end{equation}
 for $n$ large enough. 
The inequality \eqref{j6l} implies that  $\lim_{n\to\infty} \frac{4 j_6(j_6+v)}{c_{n, v}^2 l_{n, v}^2 x^2}=1.$ Replacing the term $j_6^2\log (c_{n, v} l_{n, v} x)$ in \eqref{minsum} by $j_6^2\log(4 j_6(j_6+v))/2+o(j_6^2),$ and combining alike terms, the expression \eqref{minsum} turns out to be 
$$\aligned &\quad \sum_{j=1}^{j_5}\log\mathbb{P}(X_j\ge l_{n, v} x)\\
 &=-(v^2/2+j_5 v)\log \left(1+\frac{j_5}{v}\right)+3 j_5^2/2+j_5 v\left(\frac 3 2 -\log 2\right)\\
&\quad -\,j_5\sqrt{v^2+c_{n, v}^2 l_{n, v}^2 x^2}+v j_5\log\left(1+\sqrt{1+c_{n, v}^2 l_{n, v}^2 x^2/v^2}\right) +o(j_5^2)\\
&=-\frac{v^2}{2}\log \left(1+\frac{j_5}{v}\right)-\frac12 j_5^2+\frac{1}{2}j_5 v+vj_5\log\frac{v+\sqrt{v^2+c_{n, v}^2 l_{n, v}^2 x^2}}{2(v+j_5)}+o(j_5^2)\\
&=-\frac{v^2}{2}\log \left(1+\frac{j_5}{v}\right)-\frac12 j_5^2+\frac{1}{2}j_5 v+o(j_5^2).
 \endaligned
 $$  
Here, for both the second equality and the last one, we use twice the expression  
\eqref{j6p}
and the term $o(j_6\log n)$ disappears because that the condition on $l_{n, v}$ guarantees $j_6>\!> \log n.$   
 Eventually, since $X_{(1)}=\min_{1\le i\le n} X_i,$ one gets that
$$\aligned \log\mathbb{P}(X_{(1)}\ge l_{n, v} x)&=\sum_{j=1}^{j_4}\log\mathbb{P}(X_{j}\ge l_{n, v} x)+\sum_{j=j_4+1}^n\log\mathbb{P}(X_{j}\ge l_{n, v} x)\\
&=-\frac{v^2}{2}\log \left(1+\frac{j_6}{v}\right)-\frac12 j_6^2+\frac{1}{2}j_6 v+o(j_6^2+n\log n).
 \endaligned $$
 Observing this expression, we see that the limit relating to $\log\mathbb{P}(X_{(1)}\ge l_{n, v} x)$ involves three terms $v, j_6$ and $c_{n, v} l_{n, v}.$
 
 We first work on the simpler case $\alpha\in (0, \infty],$ which indicates $c_{n, v}=O(v)$ and then $v>\!>c_{n, v} l_{n, v},$ since $l_{n, v}<\!< 1.$ Therefore, the asymptotic \eqref{j6p} tells us that 
 $$2j_6=\frac{c_{n, v}^2 l_{n, v}^2 x^2}{v+\sqrt{v^2+c_{n, v}^2 l_{n, v}^2 x^2}}+O(1)=\frac{c^2_{n, v}l^2_{n, v}}{2v} x^2+o(c^2_{n, v}l^2_{n, v}/v).$$
 Hence, $j_6=O(v l_{n, v}^2)=o(v),$ and $$\frac{n \log n}{j_6^2}=O\left(\frac{ n \log n }{v^2 l_{n, v}^4}\right)=O\left(\frac{  \log n }{n l_{n, v}^4}\right)=o(1),$$ since $l_{n, v}>\!> (\log n/n)^{1/4}$ by conditions. With the help of the Taylor expansion $\log(1+x)=x-x^2/2+o(x^2)$ for $|x|$ small enough, we have that 
$$\aligned \log\mathbb{P}(X_{(1)}\ge l_{n, v} x)=-\frac14 j_6^2+o(j_6^2).
 \endaligned $$ 
 Immediately, it follows that 
 $$\lim_{n\to\infty}\frac{1}{n^2 l_{n, v}^4}\log\mathbb{P}(X_{(1)}\ge l_{n, v} x)=-\lim_{n\to\infty}\frac{c_{n, v}^4}{64 n^2 v^2}x^4=-\lim_{n\to\infty}\frac{(n+v)^2}{4v^2}x^4=-\frac{\left(1+\alpha\right)^2}{4\alpha^2} x^4.$$ 
 Now, we consider the case that $v=\widetilde{O}(\sqrt{n\log n})$. This ensures $v<\!< c_{n, v} l_{n, v},$ and then the asymptotic \eqref{j6p} implies that 
 $$2j_6=\frac{c_{n, v}^2 l_{n, v}^2 x^2}{v+\sqrt{v^2+c_{n, v}^2 l_{n, v}^2 x^2}}+O(1)=c_{n, v}l_{n, v} x+o(c_{n, v}l_{n, v}).$$ 
 Hence, $v=o(j_6),$ and 
 $$\frac{n \log n}{j_6^2}=O\left(\frac{n \log n }{c^2_{n, v}l^2_{n, v}}\right)=O\left(\frac{  \log n }{n l_{n, v}^2}\right)=o(1)$$ since $l_{n, v}>\!> (\log n/n)^{1/2}$ by conditions.
 Then $$\log\mathbb{P}(X_{(1)}\ge l_{n, v} x)=-\frac12 j_6^2+o(j_6^2).
 $$
As a consequence, it follows 
 $$\lim_{n\to\infty}\frac{1}{n^2 l_{n, v}^2}\log\mathbb{P}(X_{(1)}\ge l_{n, v} x)=-\lim_{n\to\infty}\frac{(n+v) x^2}{2 n}=-\frac12 x^2.$$  
 We are going to discuss the last case that $\sqrt{n\log n}<\!< v<\!< n.$ Taking $l_{n, v}=v/n,$ it follows again from  \eqref{j6p} that  
 $$\lim_{n\to\infty}\frac{j_6}{v}=\lim_{n\to\infty}\frac{2(n+v)x^2}{n+\sqrt{n^2+4n(n+v) x^2}}=\frac{2x^2}{1+\sqrt{1+4x^2}}.$$
 This brings that   
 $$\lim_{n\to\infty}\frac{1}{v^2}\log\mathbb{P}(X_{(1)}\ge v x/n)=-\frac{1}{2}\log\frac{1+\sqrt{1+4x^2}}{2}-\frac{1+x^2}{2}+\frac{1}{2}\sqrt{1+4x^2}. $$ 
 For the case when $v/n<\!<l_{n, v}<\!<1,$ we have instead 
 $$\lim_{n\to\infty}\frac{j_6}{n l_{n, v}}=x.$$ 
 Therefore, it holds that
 $$\lim_{n\to\infty}\frac{1}{n^2 l_{n, v}^2}\log\mathbb{P}(X_{(1)}\ge l_{n, v} x)=-\frac{x^2}{2}.$$
 The proof is then completed. 
 
 \section{Appendix} 
 In this section, we provide the proofs of Lemmas \ref{asymlem2} and \ref{lelowerinf}.
 \begin{proof}[\bf Proof of Lemma \ref{asymlem2}] The expression \eqref{keyex} reminds us to work on the integral $\int_{c_{n, v} a_{n, v}/v}^{+\infty} e^{-v \tau_j(t)} dt$ and the factor.  
Utilize the Stirling formula 
$$\log \Gamma(z)=(z-1/2)\log z-z+O(1)$$ for $z$ large enough
to write 
\begin{equation}\label{asymz}\aligned 
	\log \frac{2^{v+1/2}v^{2j+v-1}\Gamma(j+v+1/2)}{\,\Gamma(2j+2v)\,\Gamma(j)}
	=&(2j+v)(1-\log 2)-(j-1/2)\log (j/v)\\
	&-(j+v-1/2)\log(1+j/v)+O(1)\\
\endaligned \end{equation}
for $j$ large enough, and this asymptotic still holds for $j$ bounded since both $\Gamma(j)$ and $(j-1/2)\log j$ inherit the boundedness of $j.$   
When $c_{n, v} a_{n, v}/v>x_j,$ it follows from Lemma \ref{keyl}
that 
$$\log\left( \sqrt{v}\int_{c_{n, v}a_{n, v}/v}^{+\infty} e^{-v\tau_j(t)} dt\right)\le -\log (\sqrt{v}\tau_j'(c_{n, v} a_{n, v}/v))-v\tau_j(c_{n, v} a_{n, v}/v).$$ 
By definition, 
$$\tau_j'(c_{n, v} a_{n, v}/v)=\frac{c_{n, v} a_{n, v}}{v+\sqrt{v^2+c_{n, v}^2 a_{n, v}^2}}+\frac{c_{n, v} a_{n, v} v}{2(v^2+c_{n, v}^2 a_{n, v}^2)}-\frac{2j-1}{c_{n, v} a_{n, v}}.$$
When $v>\!>c_{n, v} a_{n, v} $ or $v=O(c_{n, v} a_{n, v}),$ we see the unboundedness of  
$$\sqrt{v}\tau_j'(c_{n, v} a_{n, v}/v)=\frac{c_{n, v} a_{n, v}}{\sqrt{v}}+o\left(\frac{c_{n, v} a_{n, v}}{\sqrt{v}}\right), $$ which is due to the fact $c_{n, v} a_{n, v}/\sqrt{v}\ge 2\sqrt{n} a_{n, v}\to+\infty.$   
By conditions, we have 
$$2\sqrt{n} a_{n, v}\le c_{n, v} a_{n, v}/\sqrt{v}\le 2 n a_{n, v}.$$
This implies that
$$\log(\sqrt{v} \tau_j'(c_{n, v}a_{n, v}/v))=\tilde{O}(\log n ).$$ 	
For the case $v<\!<a_{n, v} c_{n, v},$ we see $v<\!<n$ and
$\tau_j'(c_{n, v} a_{n, v}/v)=O(1).$ Therefore 
$$\sqrt{v} \tau_j'(c_{n, v}a_{n, v}/v)=\sqrt{v}+o(\sqrt{v})\to \infty, $$ and furthermore, 
$$\log(\sqrt{v}\tau_j'(c_{n, v} a_{n, v}/v))=O(\log v)=\tilde{O}(\log n).$$ 
Also, when $j$ satisfying $c_{n, v} a_{n, v}/v>x_j,$ Lemma \ref{keyl} entails that 
$$\aligned \log\left(\sqrt{v} \int_{c_{n, v}a_{n, v}/v}^{+\infty} e^{-v\tau_j(t)} dt\right)&\ge -\log (\sqrt{v}\,\tau_j'(M))-v\tau_j(c_{n, v} a_{n, v}/v)\\
&\quad +\log(1-e^{-v \tau_j(M)+v\tau_j(c_{n, v}a_{n, v}/v)}) \endaligned $$
for any $M>c_{n, v} a_{n, v}/v.$ 
Choosing $M=c_{n, v}(1+a_{n, v})/v.$ Similarly, we see 
$$\log(\sqrt{v} \tau_j'(M))=\tilde{O}(\log n).$$ 
For the term $v \tau_j(M)-v\tau_j(c_{n, v}a_{n, v}/v),$ since $\tau_j'$ is increasing on $[x_j, \infty),$ we know 
$$v\tau_j(M)-\tau_j(c_{n, v} a_{n, v}/v)\ge \tau_j'(c_{n, v} a_{n, v}/v)(M-c_{n, v} a_{n, v}/v)= \tau_j'(c_{n, v} a_{n, v}/v) c_{n, v}.$$
The analysis above shows that $$\tau_j'(c_{n, v} a_{n, v}/v) c_{n, v}=c_{n, v}/\sqrt{v}\times \sqrt{v}\tau_j'(c_{n, v} a_{n, v}/v),$$ which tends to the positive infinity since both $\sqrt{v}\tau_j'(c_{n, v} a_{n, v}/v)$ and $c_{n, v}/v$ tend to the positive infinity. Hence, we claim that 
$$\log(1-e^{-v \tau_j(M)+v\tau_j(c_{n, v}a_{n, v}/v)})=o(1).$$ 
Therefore, 
\begin{equation}\label{asymlessb2}\aligned 
	\log \mathbb{P}(X_j\ge a_{n, v})
	= & -v\tau_j(c_{n, v} a_{n, v}/v)+(2j+v)(1-\log 2)-(j-1/2)\log (j/v)\\
	&-(j+v-1/2)\log(1+j/v)+\tilde{O}(\log n)
\endaligned \end{equation}
for $n$ large enough and for any $j$ with $c_{n, v} a_{n, v}/v>x_j.$  
Putting the expression $$\tau_j(y)=u(y)-\frac{2j-1}{v}\log y$$ into \eqref{asymlessb2}, we have that
$$\aligned \log \mathbb{P}(X_j\ge a_{n, v})&=(2j+v-1)\log v+(2j+v)(1-\log 2)-(j+v-1/2)\log(j+v)\\
 &\quad-j\log j-v \, u(c_{n, v} x/v)+(2j-1)\log(c_{n, v}x/v)+\tilde{O}(
 \log n).
\endaligned$$
Now we suppose that $c_{n, v} a_{n, v}/v\le x_j.$ 
Lemma \ref{keyl} implies the following two inequalities 
 $$\aligned \log \left(\sqrt{v}\int_{c_{n, v}a_{n, v}/v}^{+\infty} e^{-v\tau_j(t)} dt\right)&\le -v\tau_j(x_j)+\log\left(4j\right);
\\
\log \left(\sqrt{v}\int_{c_{n, v}a_{n, v}/v}^{+\infty} e^{-v\tau_j(t)} dt\right)&\ge -v\tau_j(x_j)-\log (\sqrt{v}\tau_j'(M))+\log(1-e^{-v \tau_j(M)+v\tau_j(c_{n, v}a_{n, v}/v)}) \endaligned $$
 for any $M>x_j.$  
 Choose $M=c_{n, v}(1+a_{n, v})/v,$ which verifies the condition $M>x_j,$ since $c_{n, v}a_{n, v}/v>x_j.$ 
A similar argument leads again that $\log(\sqrt{v}\tau_j'(M))=\tilde{O}(\log n )$ and $$\log(1-e^{-v \tau_j(M)+v\tau_j(c_{n, v}a_{n, v}/v)})=o(1).$$
 Therefore, we have that
 $$\log \left(\sqrt{v}\int_{c_{n, v}a_{n, v}/v}^{+\infty} e^{-v\tau_j(t)} dt\right)=-v\tau_j(x_j)+\tilde{O}(\log n ).$$
 Combining this asymptotic with the expression \eqref{asymz}, we get 
 that 
 \begin{equation}\label{asymlessb0}\aligned\log \mathbb{P}(X_j\ge a_{n, v})&=-v\tau_j(x_j)+(2j+v)(1-\log 2)-(j-1/2)\log (j/v)\\
&\quad -(j+v-1/2)\log(1+j/v)+\tilde{O}(\log n)\endaligned \end{equation} for $j$ such that 
$c_{n, v} a_{n, v}/v\le x_j.$ 
The second item of Lemma \ref{keyl} says $$4(j-3/4)(j-3/4+v)\le v^2 x_j^2\le 4(j-1/2)(j-1/2+v),$$ which implies  
$$\sqrt{v^2+v^2x_j^2}=v+2j+O(1), \quad \log(v^2 x_j^2)=\log\frac{4j(j+v)}{v^2}+o(1).$$ Thus, by definition 
$$\aligned v \tau_j(x_j)&=\sqrt{v^2+v^2x_j^2}-v \log\left(1+\sqrt{1+v^2 x_j^2}\right)-(2j-1)\log x_j\\
&=v+2j-v \log\frac{2(j+v)}{v}-(j-\frac12)\log\frac{4j(j+v)}{v^2}+O(1). \endaligned $$ 
Putting this back into \eqref{asymlessb0} and combining like terms, we get 
$$\aligned\log \mathbb{P}(X_j\ge a_{n, v})&=\tilde{O}(\log n).\endaligned$$ 
\end{proof}

\begin{proof}[\bf Proof of Lemma \ref{lelowerinf}]  The argument will be similar to that of Lemma \ref{asymlem2}. 
For the second item, we only need to prove that 
$$\log\left( \sqrt{v}\int_0^{c_{n, v}a_{n, v}/v} e^{-v\tau_j(t)} dt\right)=-v\tau_j(c_{n, v} a_{n, v}/v)+\tilde{O}(\log n).$$
Lemma \ref{keyl} guarantees 
that 
$$\log\left( \sqrt{v}\int_0^{c_{n, v}a_{n, v}/v} e^{-v\tau_j(t)} dt\right)\le -\log (-\sqrt{v}\tau_j'(c_{n, v} a_{n, v}/v))-v\tau_j(c_{n, v} a_{n, v}/v)$$ and 
$$\aligned \log\left( \sqrt{v}\int_0^{c_{n, v}a_{n, v}/v} e^{-v\tau_j(t)} dt\right)&\ge -\log(-\sqrt{v}\tau_j'(c_{n, v} a_{n, v}/(2v)))-v\tau_j(c_{n, v}a_{n, v}/v)\\
&\quad +\log\left(1-e^{v\tau_j(c_{n, v}a_{n, v}/v)-v\tau_j(c_{n, v}a_{n, v}/(2v))}\right),\endaligned $$
once 
$c_{n, v} a_{n, v}/v\le x_j.$ 
Hence, it remains to prove that 
$$\log(-\sqrt{v} \tau_j'(c_{n, v} a_{n, v}/v))=\tilde{O}(\log n), \quad e^{v\tau_j(c_{n, v}a_{n, v}/v)-v\tau_j(c_{n, v}a_{n, v}/(2v))}=o(1).$$
Indeed, using the precise form of $\tau_j,$ we have that 
$$\aligned &\quad v\tau_j(c_{n, v}a_{n, v}/v)-v\tau_j(c_{n, v}a_{n, v}/(2v))\\
&=-(2j-1)\log 2+\sqrt{v^2+c_{n, v}^2 a_{n, v}^2}-\sqrt{v^2+c_{n, v}^2 a_{n, v}^2/4}\\
&\quad +v \log\left(1+\frac{\sqrt{v^2+c_{n, v}^2 a_{n, v}^2/4}-\sqrt{v^2+c_{n, v}^2 a_{n, v}^2}}{v+\sqrt{v^2+c_{n, v}^2 a_{n, v}^2}}\right)\\
&\le -(2j-1)\log 2+\sqrt{v^2+c_{n, v}^2 a_{n, v}^2}-\sqrt{v^2+c_{n, v}^2 a_{n, v}^2/4}.
\endaligned $$
The condition that $c_{n, v} a_{n, v}\le v x_j\le 2 \sqrt{j(j+v)}$ implies the following inequality
$$ \sqrt{v^2+c_{n, v}^2 a_{n, v}^2}\le v+ 2j.$$ Since,  $\sqrt{v^2+c_{n, v}^2 a_{n, v}^2/4}\ge v,$ it follows that  
$$v\tau_j(c_{n, v}a_{n, v}/v)-v\tau_j(c_{n, v}a_{n, v}/(2v))\le-(2\log 2-1) j+\log 2\to -\infty$$
as $j\to\infty.$ Hence, 
$$e^{v\tau_j(c_{n, v}a_{n, v}/v)-v\tau_j(c_{n, v}a_{n, v}/(2v))}=o(1).$$ 
It remains to prove that
$\log(-\sqrt{v} \tau_j'(c_{n, v} a_{n, v}/v))=\tilde{O}(\log n).$ 
Recall
$$-\tau_j'\left(\frac{c_{n, v} a_{n, v}}{v}\right)=-\frac{c_{n, v} a_{n, v}}{v+\sqrt{v^2+c_{n, v}^2 a_{n, v}^2}}-\frac{c_{n, v} a_{n, v} v}{2(v^2+c_{n, v}^2 a_{n, v}^2)}+\frac{2j-1}{c_{n, v} a_{n, v}}.$$ 
When $c_{n, v} a_{n, v}>\!>v,$ it follows from the fact $c_{n, v} a_{n, v}\le v x_j\le 2\sqrt{j(j+v)}$ that  
$v<\!< j$ and $j\ge n \,a_{n, v}+o(n \,a_{n, v}).$ Thus
$$\aligned -\sqrt{v}\,\tau_j'\left(\frac{c_{n, v} a_{n, v}}{v}\right)=\frac{2j\sqrt{v}}{c_{n, v}a_{n, v}}-\sqrt{v}+o(1)=O\left(\frac{j\sqrt{v}}{n a_{n, v}}\right)\to +\infty.
\endaligned $$
On the one hand, $j\le n$ says 
$$\log \frac{j\sqrt{v}}{n a_{n, v}}\le \log \frac{\sqrt{v}}{ \,a_{n, v}}=\log\frac{\sqrt{v}\sqrt{n}}{\sqrt{n} \;a_{n, v}}=\frac12\log n+o(\log n).$$
On the other hand, it is easy to see that
$$\log \frac{j\sqrt{v}}{n a_{n, v}}\ge \log \sqrt{v}. $$
 Consequently, we have the desired expression
$$\log(-\sqrt{v} \tau_j'(c_{n,v} a_{n, v}/v))=\tilde{O}(\log n).$$ 
The same argument works when $v=O(c_{n, v} a_{n, v}).$
Now we consider the last case that $v>\!>c_{n, v} a_{n, v},$ which indicates the following lower bound of $j$ as 
$$j\ge c_{n, v}^2 a_{n, v}^2/v+o(c_{n, v}^2 a_{n, v}^2/v).$$  Under this assumption, it holds that  
$$-\sqrt{v}\tau_j'(c_{n, v} a_{n, v}/v)=-\frac{c_{n, v} a_{n, v}}{2\sqrt{v}}+\frac{2j\sqrt{v}}{c_{n,v} a_{n, v}}+o(\sqrt{n+v} \, a_{n, v})=O\left(\frac{j\sqrt{v}}{c_{n,v} a_{n, v}}\right)\to+\infty $$
as $n\to\infty.$ The last limit is due to the fact that 
$$\frac{j\sqrt{v}}{c_{n,v} a_{n, v}}\ge \frac{c_{n,v} a_{n, v}}{\sqrt{v}}\ge 2\sqrt{n} a_{n, v}\to +\infty.$$
Similarly, using the inequality $\sqrt{n} a_{n, v}\le \frac{j\sqrt{v}}{c_{n,v} a_{n, v}}\le\frac{n}{2\sqrt{n} a_{n, v}},$ we claim that
$$\log(-\sqrt{v}\tau_j'(c_{n, v} a_{n, v}/v))=\tilde{O}(\log n).$$
Now, we prove the first item, which means working on $j$ with $c_{n, v} a_{n, v}/v>x_j.$  	
As for the second item of Lemma \ref{asymlem1}, it is enough to prove that 
\begin{equation}\label{enou}
\log\left( \sqrt{v}\int_0^{c_{n, v}a_{n, v}/v} e^{-v\tau_j(t)} dt\right)=-v\tau_j(x_j)+\tilde{O}(\log n).	
\end{equation}
Indeed, Lemma \ref{estioninte} entails again that 
$$\aligned &\log\left( \sqrt{v}\int_0^{c_{n, v}a_{n, v}/v} e^{-v\tau_j(t)} dt\right)\le\log\frac{c_{n, v} a_{n, v}}{\sqrt{v}}-v\tau_j(x_j) ;\\
&\log\left( \sqrt{v}\int_0^{c_{n, v}a_{n, v}/v} e^{-v\tau_j(t)} dt\right)\ge-\log (-\sqrt{v}\tau_j'(M))+\log\left(1-e^{v\tau_j(x_j)-v\tau_j(M)}\right)-v\tau_j(x_j)
\endaligned 
$$
for any $0<M<x_j.$
It is clear that 
$$\log \frac{c_{n, v} a_{n, v}}{\sqrt{v}}=\tilde{O}(\log n).$$
Choose $M=\sqrt{j(j+4v)}/(2v).$ The remainder of the proof will be focused on 
$$\log(-\sqrt{v}\tau_j'(\sqrt{j(j+4 v)}/(2v)))=\tilde{O}(\log n) \quad \text{and} \quad v\tau_j(x_j)-v\tau_j(\sqrt{j(j+4 v)}/(2v))\to -\infty.$$ 
By definition, we have that 
$$-\sqrt{v} \tau_j'(\sqrt{j(j+4 v)}/(2v))=3\sqrt{ v j/(j+4v)}+o(\sqrt{ v j/(j+4v)})\to +\infty.$$
Hence, 
$$\log(-\sqrt{v}\tau_j'(\sqrt{j(j+4 v)}/(2v)))=\tilde{O}(\log n).$$ 
Meanwhile, it follows from the definition of $\tau_j$
and the fact $x_j>\sqrt{j(j+4  v)}/(2v)$ that $$\aligned & \quad v\tau_j(x_j)-v\tau_j(\sqrt{j(j+4  v)}/(2v))\\
&\le v\sqrt{1+x_j^2}-v-j/2-(2j-1)\log(2 v x_j/\sqrt{j(j+4v)})\\
&\le v+2j-v-j/2-(2j-1)\log(4 \sqrt{(j-1)(j-1+v)}/\sqrt{j(j+4v)})\\
&\le (3/2-2\log 2) j, \endaligned $$
which tends to $-\infty$ as $n\to\infty.$ The proof is then accomplished. 
  \end{proof}

\subsection*{Conflict of interest}  The authors have no conflicts to disclose.

\subsection*{Acknowledgment} We would like to express our sincere gratitude to the anonymous referee for their invaluable feedback and constructive comments, which significantly contributed to the improvement of this manuscript.


\begin{thebibliography}{SOSL90}
\bibitem{AS} 
M. Abramowitz and I. A. Stegun, \emph{Handbook of Mathematical Functions}, 2nd ed., Dover Publications, New York, 1972.

\bibitem{Akemann} 
G. Akemann, The complex Laguerre symplectic ensemble of non-Hermitian matrices, \emph{Nuclear Phys. B} \textbf{730} (2005), 253-299.

\bibitem{ABDbook} 
G. Akemann, J. Baik, and P. Di Francesco, \emph{The Oxford Handbook of Random Matrix Theory}, Oxford University Press, Oxford, 2011.

\bibitem{AB} 
G. Akemann and M. Bender, Interpolation between Airy and Poisson statistics for unitary chiral non-Hermitian random matrix theory, \emph{J. Math. Phys.} \textbf{51} (2010), 103524.

\bibitem{ABK} 
G. Akemann, S.-S. Byun, and N.-G. Kang, A non-Hermitian generalization of the Marchenko-Pastur distribution: from the circular law to multi-criticality, \emph{Ann. Henri. Poincar\'e} \textbf{22} (2021), 1035-1068.

\bibitem{AnG} 
G. Anderson, A. Guionnet, and O. Zeitouni, \emph{An Introduction to Random Matrices}, Cambridge University Press, Cambridge, 2010.

\bibitem{BSbook} 
Z. D. Bai and J. Silverstein, \emph{Spectral Analysis of Large Dimensional Random Matrices}, 2nd ed., Springer, New York, 2009.

\bibitem{BDSbook} 
J. Baik, P. Deift, and P. Suidan, \emph{Combinatorics and Random Matrix Theory}, American Mathematical Society, USA, 2016.

\bibitem{Bender}
M. Bender, Edge scaling limits for a family of non-Hermitian random matrix ensembles, \emph{Probab. Th. Relat. Fields} \textbf{147} (2010), 241-271.
%%%%%%%%%%%%%%
\bibitem{BForrester} 
S-S. Byun and J. Forrester, Progress on the study of the Ginibre ensembles I: GinUE, preprint (2022), arXiv:2211.16223.

\bibitem{Chafaii} D. Chafai and S. P\'ech\'e, A note on the second order universality at the edge of Coulomb gases on the plane, \emph{J. Stat. Phys.} \textbf{156} (2014), 368-383.
 
 \bibitem{JQ}S. Chang, T. Jiang, and Y. Qi, Eigenvalues of large chiral non-Hermitian random matrices, \emph{J. Math. Phys.} \textbf{61} (2020), 013508.
 %%%%%%%%%%%%
 \bibitem{Charlier03} C. Charlier,   Large gap asymptotics on annuli in the random normal matrix model. \emph{Math. Ann.} (2023), DOI: 10.1007/s00208-023-02603-z.
 \bibitem{Charlier22} C. Charlier, Asymptotics of determinants with a rotation-invariant weight and discontinuities along circles, \emph{Adv. Math.} \textbf{408} (2022), 108600.

\bibitem{DiF} 
P. Di Francesco, M. Gaudin, C. Itzykson, and F. Lesage, Laughlin's wave functions, Coulomb gases and expansions of the discriminant, \emph{Int. J. Mod. Phys. A} \textbf{9} (1994), 4257-4351.

%\bibitem{JLW} Jiang, Y., Liu, J. and Wu, L. The principle of large deviations for empirical process. \emph{ J. Math.}(chinese version), \textbf{21}(3): 295-300 (2001).
\bibitem{Fenzl}
M. Fenzl and G. Lambert, Precise deviations for disk counting statistics of invariant determinantal processes, \emph{Int. Math. Res. Not.} \textbf{2022}(10) (2022), 7420-7494.

\bibitem{Forrester} 
P. J. Forrester, \emph{Log-gases and Random Matrices}, London Mathematical Society Monographs Series \textbf{34}, Princeton University Press, Princeton, NJ, 2010.

\bibitem{Grobe} 
R. Grobe, F. Haake, and H.-J. Sommers, Quantum distinction of regular and chaotic dissipative motion, \emph{Phys. Rev. Lett.}  \textbf{61} (1988), 1899-1902.

\bibitem{Johansson}
K. Johannsson, Shape fluctuations and random matrices, \emph{Comm. Math. Phys.} \textbf{209} (2000), 437-476.

\bibitem{KSbook}
B. A. Khoruzhenko and H. J. Sommers, Non-Hermitian Random Matrix Ensembles, in: \emph{Oxford Handbook of Random Matrix Theory}, Oxford University Press, Oxford, 2011.

 \bibitem{Kostlan} 
 E. Kostlan, On the spectra of Gaussian matrices, \emph{Lin. Alg. Appl.} \textbf{162} (1992), 385-388.

\bibitem{Lacroix}
B. Lacroix-A-Chez-Toine, A. Grabsch, S. N. Majumdar, and G. Schehr, Extremes of 2D Coulomb gas: universal intermediate deviation regime, \emph{J. Stat. Mech.} (2018), 013203.

\bibitem{LACT19} 
B. Lacroix-A-Chez-Toine, J. Garzon, A. M. Calva, S. H. Castillo, I. P. Kundu, S. N. Majumdar, and G. Schehr, Intermediate deviation regime for the full eigenvalue statistics in the complex Ginibre ensemble, \emph{Phys. Rev. E} \textbf{100} (2019), 012137.

\bibitem{LACT2019}
B. Lacroix-A-Chez-Toine, S.N. Majumdar, and G. Schehr, Rotating trapped fermions in two dimensions and the complex Ginibre ensemble: Exact results for the entanglement entropy and number variance, \emph{Phys. Rev. A} \textbf{99} (2019), 021602.

\bibitem{MP}
V. A. Marchenko and L. A. Pastur, Distributions of some sets of random matrices, \emph{Math. USSR-Sb} \textbf{1} (1967), 457-483.

\bibitem{MaLDP} Y. T. Ma, Unified limits and large deviation principles for $\beta$-Laguerre ensembles in global regime, \emph{Acta Math. Sin.} \textbf{39} (2013), 1271-1288.

\bibitem{Olver} 
F. W. J. Olver, D. W. Lozier, R. F. Boisvert, and C. W. Clark, \emph{NIST Handbook of Mathematical Functions}, Cambridge University Press, Cambridge, 2010.

\bibitem{osborn} 
J. C. Osborn, Universal results from an alternate random matrix model for QCD with a baryon chemical potential, \emph{Phys. Rev. Lett.} \textbf{93} (2004), 2220019.

\bibitem{Rider03} 
B.C. Rider, A limit theorem at the edge of a non-Hermitian random matrix ensemble, \emph{J. Phys. A} \textbf{36} (2003), 3401-3409.

\bibitem{Rider04} B.C. Rider, Order statistics and Ginibre' s ensembles, \emph{ J. Stat. Phys.} \textbf{114} (2004), 1139-1148.

\bibitem{Rider 14} B.C. Rider and C.D. Sinclair, Extremal laws for the real Ginibre ensemble, \emph{ Ann. Appl. Probab.} \textbf{24}(4) (2014), 1621-1651.
\bibitem{Stephanov} M. A. Stephanov, Random matrix model for qcd at finite density and the nature of the quenched limit.,\emph{ Phys. Rev. Lett.} \textbf{76} (1996), 4472-4475. 

\end{thebibliography}
\end{document}